\documentclass{mcom-l}

      \usepackage{amssymb}
      \usepackage{amsmath}
      \usepackage{graphicx}
      \usepackage{multirow}
      \usepackage{adjustbox}
      \usepackage[boxed,noend]{algorithm2e}
      \usepackage{hyperref}
      \usepackage{xcolor}
      \usepackage{caption}
      \usepackage{placeins}
      \usepackage{relsize}

      \theoremstyle{plain}
      \newtheorem{theorem}{Theorem}[section]
      \newtheorem{lemma}[theorem]{Lemma}
      
      \newtheorem{proposition}[theorem]{Proposition}

      \theoremstyle{definition}

      \theoremstyle{remark}

      \newcommand{\smalltimes}{\hspace{0.5pt}\raisebox{0.25ex}{$\mathsmaller{\times}$}}

      \makeatletter
      \@namedef{subjclassname@2020}{\textup{2020} Mathematics Subject Classification}
      \makeatother

      \copyrightinfo{}{}
      \makeatletter
      \let\article@logo\@empty
      \def\@logofont{\fontsize{6}{7pt}\selectfont}
      \let\@serieslogo\@empty
      \makeatother

\begin{document}

   \author{Greg Hurst}
   \address{Bedford, New Hampshire}
   \email{ghurst588@gmail.com}

   \title[Computing the Mertens Function]{Practical Computations of the Mertens Function: $M(10^{24})$ and $M(10^{25})$}

   \date{}

   \begin{abstract}
     The Mertens function is defined as $M(x)=\sum_{n\leq x}\mu(n)$, where $\mu(n)$ is the M\"obius function.
     This paper describes a practical implementation of the classical $O(x^{2/3+\varepsilon})$ algorithm for computing $M(x)$ at isolated values, together with the segmented M\"obius and Mertens sieve on which it relies.
     The implementation was used to compute
     $$ M(10^{24}) = 7\,189\,337\,839 \quad\text{and}\quad M(10^{25}) = -258\,560\,632\,948, $$
     taking $7.0$ days and $34.6$ days, respectively.
     These computations extend the previous record of $M(10^{23})$ by two orders of magnitude.
     Run standalone, the segmented sieve computed all Mertens values through $10^{16}$ in approximately $7.4$ days, compared with the $7.5$-month runtime of the author's 2018 computation.
     In pursuit of the most practical isolated-value method, an optimized implementation of the asymptotically faster Helfgott-Thompson algorithm is also presented, running roughly four times faster than the original implementation.
     With both implementations optimized, the choice between methods depends on input size and hardware.
     The contribution is a reproducible computation and implementation study, rather than a new asymptotic algorithm.
   \end{abstract}

   \maketitle

   \section{Introduction}
   \label{intro}

   The M\"obius function is defined by
   $$ \mu(n) = \begin{cases}(-1)^{\omega(n)} & \mbox{ if } n \mbox{ is square-free},\\ 0&\mbox{ otherwise},\end{cases} $$
   where $\omega(n)$ is the number of distinct prime factors of $n$.
   Its summatory function
   $$ M(x)=\sum_{n\leq x}\mu(n) $$
   is known as the Mertens function.
   The size and oscillation of $M(x)$ are closely connected with the distribution of prime numbers.
   For example, for $\operatorname{Re}(s)>1$,
   $$ \frac{1}{\zeta(s)}=s\int_1^\infty M(x)x^{-s-1}\,dx. $$
   Thus the growth of $M(x)$ is tied to the distribution of the zeros of $\zeta(s)$.
   In particular, $M(x)=O(x^{1/2+\varepsilon})$ is equivalent to the Riemann hypothesis.

   Two computational problems associated with $M(x)$ are relevant here.
   The first is to compute every value $M(n)$ up to a given endpoint.
   This gives detailed information about the behavior of the function over an interval, but requires at least linear work in the endpoint.
   The second is to compute $M(x)$ at one isolated value.
   The two problems are related, as the isolated-value algorithms require a sieve over a substantially shorter initial interval.

   This paper reports the isolated values
   $$ M(10^{24}) = 7\,189\,337\,839 \quad\text{and}\quad M(10^{25}) = -258\,560\,632\,948. $$
   The largest previously reported isolated value was $M(10^{23})$, computed by Helfgott and Thompson \cite{HT}, so these values extend the record by two orders of magnitude.
   The new computations took $7.0$ days and $34.6$ days, respectively, on the machine described in Section \ref{isolatedresults}.
   The $10^{25}$ computation required sieving through $13\,694\,622\,981\,236\,974$, slightly beyond the $10^{16}$ endpoint of the full interval computation in \cite{Hur18}.

   The isolated-value calculations above use the classical $O(x^{2/3+\varepsilon})$ framework of \cite{DR}, \cite{EK}, and \cite{Hur18}.
   This is not the best currently known exponent.
   Helfgott and Thompson introduced an elementary algorithm with runtime $O(x^{3/5}(\log x)^{3/5+\varepsilon})$, and used it to set the previous record \cite{HT}.
   Subsequent work gives further elementary asymptotic improvements for the Mertens function \cite{HKM}.
   The contribution here is not a new exponent.
   It is a reproducible practical computation at a new scale, together with the implementation choices and validation needed to make the result checkable.

   The computations above rely on several implementation choices, in particular inclusion-exclusion reductions that decrease the number of visited summands, a parallel bucket scheduler for large primes, and compressed Mertens storage.
   Together, these choices make the larger isolated-value computations possible.
   They also substantially improve the standalone segmented sieve, which computes all Mertens values through $10^{16}$ in approximately $7.4$ days rather than the $7.5$ months reported in \cite{Hur18}.
   The values of $\mu(n)$ and $M(n)$ appear throughout number theory, and many applications need them over long intervals rather than at isolated points.
   The sieve is therefore also released as a reusable tool.

   The paper also includes a comparison with the asymptotically superior Helfgott-Thompson algorithm.
   Both methods have set records, and both are practical at scale.
   Each implementation is optimized in its own right, and the aim is to understand which method to prefer at reachable sizes.
   Since the two implementations stress different parts of the hardware, the answer could depend on the machine as well as the input.

   Section \ref{history} reviews previous interval and isolated-value computations.
   Section \ref{combinatorial} derives the classical formula used here, and Section \ref{ie} gives the inclusion-exclusion reductions.
   Section \ref{sieve} describes the segmented M\"obius and Mertens sieve.
   Sections \ref{algorithm} and \ref{implementation} describe the phase structure and the principal implementation choices.
   Section \ref{results} gives the computed values and timings, and Section \ref{validation} describes the validation checks.
   Section \ref{htcompare} compares the practical behavior of the classical and Helfgott-Thompson implementations as a hardware-dependent algorithm selection problem, and Appendix \ref{gpuvariant} describes experimental GPU variants of both isolated-value implementations.
   Section \ref{codeavailability} describes the public code release.

   \section{Historical Context}
   \label{history}

   Computations of the Mertens function fall naturally into interval computations and isolated-value computations.
   In an interval computation, every value $M(n)$ up to an endpoint is produced.
   Mertens computed $M(n)$ by hand through $10^4$ in 1897 and, based on that table, conjectured that $|M(x)|<\sqrt{x}$ for all $x>1$ \cite{Mer}.
   Von Sterneck continued the tabulations in a series of papers, eventually reaching selected values up to $5\cdot 10^6$, and proposed the stronger bound $|M(x)|<\frac{1}{2}\sqrt{x}$ for $x>200$ \cite{St1}, \cite{St2}.

   Machine computation transformed the subject.
   Neubauer computed all values through $10^8$ in 1963 and, sampling beyond that range, disproved von Sterneck's bound near $7.76\cdot 10^9$ \cite{Neu}.
   Cohen and Dress computed through $7.8\cdot 10^9$, locating the smallest such counterexample \cite{CD}.
   The Mertens conjecture itself was disproved by Odlyzko and te Riele in 1985, using analytic methods, notably without producing an explicit counterexample \cite{OtR}.
   Their techniques were later revisited in \cite{Hur18} to improve the associated bounds on $M(x)/\sqrt{x}$.
   Dress extended the interval computations to $10^{12}$ \cite{Dre}, and Lioen and van de Lune reached $10^{13}$ with vectorized sieving \cite{LL}.
   Kotnik and van de Lune computed all values through $10^{14}$ in 2003 \cite{KL}.
   This was extended to $10^{16}$ in 2018 \cite{Hur18}, recording all extrema, zeros, and regularly spaced values.
   The segmented sieve used in the present work is descended from that implementation, but the addition of a parallel bucket scheduler for large primes substantially changes its behavior on long intervals.

   The isolated-value problem has a different history.
   Rather than producing every value below $x$, isolated-value methods rest on combinatorial identities that need only a small fraction of the M\"obius values below $x$.
   Lehman gave an early version of the classical method in 1960 \cite{Leh}.
   In 1996, Del\'eglise and Rivat gave an algorithm with runtime $O(x^{2/3}(\log\log x)^{1/3})$ and space $O(x^{1/3}(\log\log x)^{2/3})$ \cite{DR}, and used it to compute
   $$ M(10^{16})=-3\,195\,437. $$
   The first improvement came in 2011, when Kuznetsov used a GPU variant of the same classical algorithm to compute all powers of $10$ through $10^{22}$ \cite{EK}, reporting
   $$ M(10^{22})=-2\,061\,910\,120. $$
   The computations in \cite{Hur18} also used a version of this algorithm for isolated values, including powers of two through $2^{73}$, falling just short of $10^{22}$.

   Helfgott and Thompson later introduced a new elementary algorithm with runtime $O(x^{3/5}(\log x)^{3/5+\varepsilon})$ and space $O(x^{3/10}(\log x)^{13/10})$.
   As noted in \cite{HT}, this was the first improvement in the exponent of $x$ for an elementary algorithm since related combinatorial work on prime counting in 1985.
   They reported computations of $M(x)$ for $x=2^n$, $n\leq 75$, and in 2021 computed all powers of $10$ through $10^{23}$, including
   $$ M(10^{23})=62\,467\,771\,689. $$

   There are also asymptotically faster directions beyond the two implementations explored in this paper.
   Hirsch, Kessler, and Mendlovic described an elementary framework for computing the prime counting function $\pi(x)$ in $O(x^{1/2+\varepsilon})$ time and applied the same techniques to the Mertens function \cite{HKM}.
   For $M(x)$, their space can be reduced to $O(x^{1/3+\varepsilon})$ while the runtime rises to $O(x^{8/15+\varepsilon})$.
   Analytic methods related to the Lagarias-Odlyzko algorithm for prime counting give another asymptotically favorable direction \cite{LO}.
   To the author's knowledge, neither approach has been implemented for $M(x)$.

   The present work focuses on the practical side of isolated-value computation, studying how far the range can be pushed when every stage, from summation identities down to parallel scheduling, is designed as one computation.

   \section{Combinatorial Algorithm}
   \label{combinatorial}

   Throughout the remainder of the paper, arguments obtained by division are implicitly rounded down.
   For example, $M(x/n)$ denotes $M(\lfloor x/n\rfloor)$.
   The algorithmic setup starts with the well-known identity \cite{Leh}
   $$ \sum_{n \leq x} M(x/n) = 1. $$
   Grouping together terms with the same value of $\lfloor x/n \rfloor$, the identity becomes
   \begin{align*}
     \sum_{n \leq \kappa_x} M(x/n)
        &= 1 - \sum_{n \leq \nu_x} \left( \left\lfloor \frac{x}{n} \right\rfloor - \left\lfloor \frac{x}{n+1} \right\rfloor \right) M(n) \\
        &= 1 + \kappa_x M(\nu_x) - \sum_{n \leq \nu_x} \left\lfloor \frac{x}{n} \right\rfloor \mu(n),
   \end{align*}
   where
   $$ \nu_x = \lfloor \sqrt{x}\, \rfloor, \quad \kappa_x = \left\lfloor \frac{x}{\nu_x+1} \right\rfloor. $$
   The form of the identity involving $\mu(n)$ is computationally preferable both because the M\"obius values can be stored in one byte, and the zero terms can be skipped.

   As in \cite{Hur18}, fix $u$ and for an argument $y$ write
   $$ S_1(y,u)=\sum_{y/u < n \leq \kappa_y} M(y/n),
      \qquad
      S_2(y)=\sum_{n \leq \nu_y} \left\lfloor \frac{y}{n} \right\rfloor \mu(n). $$
   For an argument $y$ with $\nu_y < u$, define
   \begin{equation}
   \label{Sdef}
     S(y,u)=1-S_1(y,u)+\kappa_yM(\nu_y)-S_2(y).
   \end{equation}
   Then, for $\nu_x < u < x$,
   \begin{equation}
   \label{Ssumidentity}
      \sum_{k \leq x/u} M(x/k) = S(x,u).
   \end{equation}
   Applying generalized M\"obius inversion gives the following formula.

   \begin{theorem}
   \label{mainformula}
   For $\nu_x < u < x$,
   $$ M(x) = \sum_{k \leq x/u} \mu(k)S(x/k,u). $$
   \end{theorem}

   The cost of evaluating Theorem \ref{mainformula} is the cost of sieving up to $u$, together with the cost of evaluating the sums in $S(x/k,u)$ for $k\leq x/u$.
   Ignoring logarithmic factors, the total work is
   $$ O\bigg(u + \sum_{k \leq x/u} \sqrt{x/k}\bigg) = O(u + x/\sqrt{u}). $$
   Balancing these terms gives $u=O(x^{2/3})$ and total runtime $O(x^{2/3+\varepsilon})$.
   Since only square-free $k$ contribute to the outer sum, about $6/\pi^2\approx 60.8\%$ of the outer terms remain.

   The split point $\nu_y$ in (\ref{Sdef}) need not be exactly $\lfloor\sqrt{y}\rfloor$.
   More generally, one may choose any integer split point $1\leq \nu_y<y$ and define $\kappa_y=\lfloor y/(\nu_y+1)\rfloor$ as before.
   In the implementation below, $\nu_y$ is biased by a constant factor relative to $\sqrt{y}$.
   This changes the amount of work assigned to $S_1$ and $S_2$, but not the underlying identity.
   Increasing $\nu_y$ adds terms to $S_2(y)$ and removes terms from $S_1(y,u)$.
   Similarly, the main sieve bound $u$ is chosen as a constant multiple of the asymptotic balance.
   The constants are tunable parameters, chosen to balance the costs of sieving, quotient computation, and memory traffic.

   \section{Inclusion-Exclusion Reductions}
   \label{ie}

   The formula in Theorem \ref{mainformula} can be implemented directly, but many of its terms carry redundant work.
   Its outer sum and the sum in $S_2(y)$ are both weighted by the M\"obius function.
   A direct implementation already skips indices for which $\mu(n)=0$, but it still visits many nonzero terms whose contributions can be grouped according to divisibility by the first few primes.
   Inclusion-exclusion makes the resulting cancellations explicit.
   The resulting formulas restrict the indices that must be visited and replace several nearby terms by a single simpler summand.
   As estimated below, these reductions decrease both main summation components by a large amount.
   The general identity is the following.

   \begin{lemma}
   \label{ielemma}
   Let $P$ be a finite set of primes and let $Q=\prod_{p\in P}p$.
   For any arithmetic function $f$ with $f(n)=0$ for $n>N$,
   $$ \sum_{n \leq N}\mu(n)f(n)
      =
      \sum_{d\mid Q}\mu(d)
      \sum_{\substack{m \leq N/d\\(m,Q)=1}}\mu(m)f(dm). $$
   \end{lemma}

   This follows immediately from inclusion-exclusion on $P$: every square-free $n$ factors uniquely as $n=dm$, where $d\mid Q$ and $(m,Q)=1$, giving $\mu(n)=\mu(d)\mu(m)$.

   The two cases used below are
   \begin{equation}
   \label{formula2}
     \sum_{n \leq N} \mu(n)f(n)
       =
       \sum_{\substack{n \leq N\\(n,2)=1}}\mu(n)(f(n)-f(2n))
   \end{equation}
   and
   \begin{equation}
   \label{formula6}
     \sum_{n \leq N} \mu(n)f(n)
       =
       \sum_{\substack{n \leq N\\(n,6)=1}}\mu(n)(f(n)-f(2n)-f(3n)+f(6n)).
   \end{equation}
   To use these reductions with functions that do not vanish past the summation range, first define
   $$ f_N(n)=
      \begin{cases}
      f(n), & 1\leq n\leq N,\\
      0, & n>N.
      \end{cases} $$
   Applying Formula (\ref{formula2}) to this truncated function, note that $f_N(n)=f(n)$ throughout the summation range, while $f_N(2n)=f(2n)$ only for $n\leq \lfloor N/2\rfloor$.
   Hence, for any function $f$,
   \begin{equation}
   \label{formula2split}
     \sum_{n\leq N}\mu(n)f(n)
     =
     \sum_{\substack{n\leq \lfloor N/2\rfloor\\(n,2)=1}}\mu(n)(f(n)-f(2n))
     +
     \sum_{\substack{\lfloor N/2\rfloor<n\leq N\\(n,2)=1}}\mu(n)f(n).
   \end{equation}
   The same truncation scheme is used in the other reductions below.
   The two sums in (\ref{formula2split}) are called the \emph{combined range} and the \emph{single range}, respectively.

   \begingroup
   \setlength{\abovedisplayskip}{5pt}
   \setlength{\belowdisplayskip}{5pt}
   \setlength{\abovedisplayshortskip}{3pt}
   \setlength{\belowdisplayshortskip}{4pt}

   With this notation, the $S_1$ part of the outer computation is
   $$ \sum_{k\leq x/u}\mu(k)S_1(x/k,u), $$
   and applying (\ref{formula2split}) transforms this sum into
   $$ \sum_{\substack{k\leq x/(2u)\\(k,2)=1}}\mu(k)(S_1(x/k,u)-S_1(x/(2k),u))
      +
      \sum_{\substack{x/(2u)<k\leq x/u\\(k,2)=1}}\mu(k)S_1(x/k,u). $$
   In the combined range of this transformation, write $y=x/k$.
   Rewriting $S_1(y/2,u)$ with an even index cancels the common terms, giving
   \begin{equation}
   \label{s1parity}
   S_1(y,u)-S_1(y/2,u)
    =
    \sum_{\substack{y/u<n\leq \kappa_y\\(n,2)=1}}M(y/n)
    -
    \sum_{\substack{\kappa_y<n\leq 2\kappa_{y/2}\\2\mid n}}M(y/n).
   \end{equation}
   This reduces the number of terms in two ways.
   First, the outer summation over $k$ is restricted to $(k,2)=1$.
   Second, in the combined range, the middle part of the $S_1$ sum cancels, and the two remaining pieces each visit only every other index.

   The sum involving $S_2$ has two separate reductions.
   One is applied inside the definition of $S_2$ itself.
   The other is applied to the outer sum, similar to the $S_1$ reduction above.
   For the inner reduction, the truncated version of Formula (\ref{formula6}) is applied, analogous to (\ref{formula2split}).
   Writing $j$ for the inner summation index, this restricts $j$ to values coprime to $6$ and replaces the $\lfloor y/j\rfloor$ term in the summand by
   \begin{equation}
   \label{s2coprime6}
   \left\lfloor \frac{y}{j} \right\rfloor
   - \left\lfloor \frac{y}{2j} \right\rfloor
   - \left\lfloor \frac{y}{3j} \right\rfloor
   + \left\lfloor \frac{y}{6j} \right\rfloor.
   \end{equation}
   The original $S_2$ sum only includes indices up to $\nu_y$.
   Thus the terms in (\ref{s2coprime6}) with $2j$, $3j$, and $6j$ stop contributing once their indices exceed $\nu_y$.
   This introduces breakpoints at the corresponding cutoff values $\nu_y/2$, $\nu_y/3$, and $\nu_y/6$.

   The outer $S_2$ sum is then split by Formula (\ref{formula2split}) as before.
   Writing $y=x/k$, the combined range contains $S_2(y)-S_2(y/2)$, while the single range contains only $S_2(y)$.
   In the combined range, the breakpoints for $S_2(y)$ and $S_2(y/2)$ are merged.
   Let $A=\nu_y$ and $B=\nu_{y/2}$.
   The values to be sorted are
   $$ A/6,\ A/3,\ A/2,\ A
      \quad\text{and}\quad
      B/6,\ B/3,\ B/2,\ B, $$
   Since $B\approx A/\sqrt{2}$, the merged values interleave in a fixed order, as shown in the tables below.

   On each resulting interval the summand is an elementary function of $q=\lfloor y/j\rfloor$.
   The two tables give the formulas used in the implementation for the combined and single ranges.
   These are exact simplifications of the inclusion-exclusion summands, with each range restricted to $j$ coprime to $6$.

   \vspace{7pt}
   {\fontsize{8}{8}\selectfont
   \setlength{\tabcolsep}{2pt}
   \renewcommand{\arraystretch}{1.3}
   \par\noindent
   \begin{minipage}[t]{0.64\textwidth}
   \centering
   \begin{adjustbox}{max width=\linewidth}
   \begin{tabular}{|r@{}c@{}l|l|}
   \hline
   \multicolumn{3}{|c|}{range of $j$} & \multicolumn{1}{c|}{summand}\\
   \hline
   $1\leq{}$ & $j$ & ${}\leq B/6$ & $\lfloor q/4\rfloor+(q\bmod 2)-\lfloor q/12\rfloor-(\lfloor q/3\rfloor\bmod 2)$\\
   $B/6<{}$ & $j$ & ${}\leq A/6$ & $\lfloor q/4\rfloor+(q\bmod 2)-[q\bmod 6\geq 3]$\\
   $A/6<{}$ & $j$ & ${}\leq B/3$ & $\lfloor q/4\rfloor+(q\bmod 2)-\lfloor(q+3)/6\rfloor$\\
   $B/3<{}$ & $j$ & ${}\leq A/3$ & $\lfloor q/4\rfloor+(q\bmod 2)-\lfloor q/3\rfloor$\\
   $A/3<{}$ & $j$ & ${}\leq B/2$ & $\lfloor q/4\rfloor+(q\bmod 2)$\\
   $B/2<{}$ & $j$ & ${}\leq A/2$ & $q\bmod 2$\\
   $A/2<{}$ & $j$ & ${}\leq B$ & $\lceil q/2\rceil$\\
   $B<{}$ & $j$ & ${}\leq A$ & $q$\\
   \hline
   \end{tabular}
   \end{adjustbox}
   \par\vspace{2pt}
   combined range: $S_2(y)-S_2(y/2)$
   \end{minipage}\hfill
   \begin{minipage}[t]{0.33\textwidth}
   \centering
   \begin{adjustbox}{max width=\linewidth}
   \begin{tabular}{|r@{}c@{}l|l|}
   \hline
   \multicolumn{3}{|c|}{range of $j$} & \multicolumn{1}{c|}{summand}\\
   \hline
   $1\leq{}$ & $j$ & ${}\leq A/6$ & $\lceil q/2\rceil-\lceil\lfloor q/3\rfloor/2\rceil$\\
   $A/6<{}$ & $j$ & ${}\leq A/3$ & $\lceil q/2\rceil-\lfloor q/3\rfloor$\\
   $A/3<{}$ & $j$ & ${}\leq A/2$ & $\lceil q/2\rceil$\\
   $A/2<{}$ & $j$ & ${}\leq A$ & $q$\\
   \hline
   \end{tabular}
   \end{adjustbox}
   \par\vspace{2pt}
   single range: $S_2(y)$
   \end{minipage}\par
   }
   \endgroup
   \vspace{7pt}

   A limiting density count of the surviving summands gives a useful estimate of the savings, without modeling the cost of each summand.
   Relative to a direct implementation of Theorem \ref{mainformula} that already skips terms with $\mu(n)=0$ wherever $\mu$ appears, this predicts a reduction by a factor of $3/(3-\sqrt2)\approx 1.89$ in the $S_1$ part and by a factor of $3$ in the $S_2$ part.
   These factors come from the summands skipped by the parity and coprime-to-$6$ restrictions, together with the cancellation in the paired range of (\ref{s1parity}).
   For the split points used in the computations, the two components together give a summand-count reduction just below a factor of $2.5$.

   \section{Segmented M\"obius and Mertens Sieving}
   \label{sieve}

   The computation requires repeated access to $\mu(n)$ and $M(n)$ on successive segments of integers up to the sieve bound $u$.
   These values are produced by a segmented M\"obius sieve followed by a segmented prefix sum.
   The same implementation also serves as a standalone tool for producing long intervals of $\mu(n)$ and $M(n)$.
   Internally, the sieve partitions each segment into fixed-size pieces called \emph{sub-segments}, chosen so that the simultaneously active sub-segments across threads fit in cache.
   It has four main components: byte-encoded M\"obius sieving, a stencil for the first primes, bucket scheduling for large primes, and a compressed prefix sum for the resulting Mertens values.
   The first two descend from the implementation in \cite{Hur18}: the byte-logarithm encoding is essentially the same method, and the stencil uses the same pre-sieved local structure.
   The new practical contributions are the parallel bucket scheduler, the compressed Mertens prefix representation at the scales used here, and the tuning of segment sizes and parallel work distribution around these choices.
   For large inputs, the bucket scheduler is the main driver of the sieve's performance improvement.
   Below, \emph{small primes} are those handled by unrolled code after the stencil is copied, \emph{medium primes} are the remaining primes sieved directly, and \emph{large primes} are those handled by the bucket scheduler.

   \subsection{Byte-Encoded M\"obius Sieving}

   The M\"obius sieve stores one signed byte per value in the current segment during the sieving phase.
   Rather than storing the product of prime factors, it stores a compact approximation to the sum of logarithms of prime factors, together with parity information.
   As in \cite{EK} and \cite{Hur18}, a logarithmic weight is added to each multiple of $p$ in the sieve range, and multiples of $p^2$ are marked as zero.
   Here the weight is $\lceil \log_2 p \rceil \,|\, 1$, where $|$ denotes bitwise OR.
   Those works use $\lfloor \log_2 p \rfloor$, whereas taking the ceiling makes the finalization comparison exact and removes the encoding collision discussed below.
   The stored value is later finalized by comparison with the size of the integer being sieved.

   \subsection{Stencil and Unrolled Small Primes}

   The first primes and prime powers are handled by a precomputed stencil of period $13860=2^2\cdot 3^2\cdot 5\cdot 7\cdot 11$, which is copied into each new segment.
   This copy is done in large tiles so that it runs at high memory bandwidth.
   After the stencil has been copied, primes up to $353$ are applied through unrolled code.

   \subsection{Large-Prime Scheduling}

   Large primes require additional treatment.
   For a sub-segment being processed, most primes larger than its length do not hit that sub-segment at all.
   With fixed-size sub-segments, iterating through all such primes for every sub-segment wastes time and can raise the sieving cost from near-linear to about $O(u^{3/2})$.
   The bucket scheduler avoids rediscovering these sparse hits by carrying each large prime forward from one contributing sub-segment to the next.
   This follows a suggestion in the concluding remarks of \cite{Hur18}, where bucket scheduling was proposed as a way to extend the interval computation beyond $10^{16}$.

   Each large prime is stored in the bucket corresponding to the next sub-segment that contains a multiple of that prime.
   When a sub-segment is processed, each stored prime is applied at its next hit and then forwarded to the bucket of its following hit.

   The buckets are stored in a circular buffer, with a power-of-two number of buckets.
   If the sub-segment length is $L$ and the circular buffer has $B$ buckets, then the scheduler must satisfy $\sqrt{u} < BL$,
   since $\sqrt{u}$ is the largest prime that can appear in the segmented sieve.
   The power-of-two choice keeps the bucket index arithmetic simple, since wraparound can be handled by masking instead of integer division.

   \subsection{Finalization}

   To finalize the M\"obius values, let $S=\sum(\lceil\log_2 p\rceil\,|\,1)$ be the sum accumulated for a square-free $n$.
   Then $n$ is fully factored within the sieve when $S>\lfloor\log_2 n\rfloor$, and otherwise retains a single unsieved prime cofactor. Since every log value has its least significant bit set to $1$, the sign follows from the parity of $S$, giving $\mu(n)=1-2(S\bmod 2)$ when $S>\lfloor\log_2 n\rfloor$ and $2(S\bmod 2)-1$ otherwise.
   The threshold $\lfloor\log_2 n\rfloor$ is constant between consecutive powers of two, so finalization is applied separately on each such interval.
   Under the ceiling weight used here, this comparison is exact.
   Each weight is at most $2\log_2 p$, so an unsieved prime cofactor $q>\sqrt{u}$, which exceeds $n/q$, gives $S\leq 2\log_2(n/q)\leq \log_2 n$ and hence $S\leq\lfloor\log_2 n\rfloor$.
   Conversely, a fully sieved square-free $n>2$ has $S\geq\lceil\log_2 n\rceil>\lfloor\log_2 n\rfloor$.
   Unlike the floor encoding of \cite{Hur18}, whose first known collision is near $1.15\cdot 10^{18}$, this places no nontrivial bound on the sieve endpoint.
   The only requirement is that the accumulated byte not overflow its seven value bits, which holds for every $u<2^{64}$.
   Finalization is performed per sub-segment immediately after the large-prime hits are applied, keeping the values in-cache, instead of a separate pass over the segment.
   Non-square-free $n$ already have $\mu$ marked as $0$, and $\mu(1)$, $\mu(2)$ are set directly.

   \subsection{Compressed Mertens Storage}

   Once $\mu(n)$ has been computed over an interval, the corresponding values of $M(n)$ are obtained by a prefix sum and stored in compressed form.
   The compressed representation is
   $$ M(t) = C\!\left(\left\lfloor\frac{t-L}{H}\right\rfloor\right) + R(t-L), $$
   where $L$ is the first index in the current interval, $H$ is a power-of-two stride, and $C$ and $R$ are the \emph{coarse} and \emph{residual} arrays.
   The array $C$ stores a full Mertens value once every $H$ positions, while $R$ stores the signed byte offset from this coarse value.
   Thus a lookup of $M(t)$ requires access into two arrays.
   The stride $H$ can be tuned to the intended sieve range: larger values improve compression and cache behavior, while longer sieved ranges may require a smaller stride so that the byte residual remains valid.
   Choosing $H$ to be a power of two also keeps the coarse array index computation cheap.
   The signed-byte residual requires the change in $M$ over a stride-$H$ interval to remain in $[-128,127]$.

   The prefix sum from $\mu$ to this representation is done in three stages.
   First, within each interval of length $H$, the values of $\mu$ are prefix summed locally, producing the residual values relative to the start of that interval.
   This local scan is vectorized.
   Second, the total sum from each such interval is itself prefix summed in parallel.
   Third, these interval totals, together with the incoming value of $M$ from the previous segment, determine the coarse values stored in $C$.
   When only the compressed Mertens values are needed, the byte-logarithm finalization is folded into this scan, without materializing the values of $\mu$ all at once.

   \section{Isolated-Value Algorithm}
   \label{algorithm}

   The isolated-value computation uses the sieved values of $\mu$ and $M$ to evaluate the sum in Theorem \ref{mainformula}.

   \subsection{Parameter Choices}
   \label{parameterchoices}

   The identities above leave the split points and the sieve bound as the main tuning choices.
   For a given input $x$, the sieve bound is set to
   $$ u = \left\lceil f_x\cdot (x/\log\log x)^{2/3} \right\rceil, $$
   where $f_x=1.05$ for $x\le 10^{16}$, decreases linearly in $\log_{10}x$ to $0.75$ at $x=10^{22}$, and is held fixed beyond that point.
   For each argument $y$ in the calls to $S(y,u)$, the split point in (\ref{Sdef}) is
   $$ \nu_y = \lfloor c_x\sqrt{y}\rfloor, $$
   with the corresponding $\kappa_y=\lfloor y/(\nu_y+1)\rfloor$.
   The coefficient $c_x$ is held fixed at $1.5$.

   These choices are empirical and should be retuned on other machines.
   Broadly speaking, decreasing $f_x$ reduces the sieve and bucket scheduler cost, but leaves more work for summation terms that increasingly use the 128-bit quotient path.
   Increasing $c_x$ shifts work from $S_1$ to $S_2$.
   This can help because $S_1$ performs noncontiguous lookups of Mertens values, while $S_2$ streams through contiguous M\"obius values.
   The shift is useful only for modest changes around the square-root split, since otherwise $S_2$ itself becomes too long.
   Thus $c_x$ gives a narrower adjustment, shifting work inside each $S(y,u)$ evaluation without changing the sieve cost.
   The sieve bound factor $f_x$ is the broader tuning parameter.

   \subsection{Square-Free Indices and Stored Entries}

   The reductions of Section \ref{ie} leave only odd square-free indices $k\leq x/u$.
   For each such index, set $x_k=\lfloor x/k\rfloor$.
   The computation stores one running value for $S(x_k,u)$, initially set to $1$.
   A preliminary M\"obius sieve identifies these indices, and the code stores only the nonzero terms in the main arrays.
   Two index arrays, each the inverse of the other, translate between the original index $k$ and its compact position among the stored indices.

   Along with $x_k$, the values $\nu_{x_k}$, $\kappa_{x_k}$, $\lfloor x_k/u\rfloor$, and $2\kappa_{x_k/2}$ are stored.
   The last quantity is used in the parity-reduced $S_1$ range in Formula (\ref{s1parity}).
   Most entries in these arrays fit in 64 bits, and the few that do not are stored separately in 128-bit arrays.
   This separation keeps most of the computation on ordinary 64-bit arithmetic, while the 128-bit path is reserved for the largest values of $x_k$.

   \subsection{Main Computation}

   The main part of the computation updates the running values as the sieve advances.
   For each segment, the relevant inclusion-exclusion ranges from Section \ref{ie} are applied, subtracting the $S_1$ and $S_2$ contributions and adding the $\kappa M(\nu)$ term.
   This is organized into two phases.
   The phase split is part of the algorithm, while the storage widths and segment sizes used inside each phase are implementation choices.

   The first phase sieves only as far as needed for $\nu_x$.
   During this phase, both $\mu$ and $M$ are needed.
   The coarse Mertens values are stored as 16-bit integers until just before $|M(n)|$ first exceeds $2^{15}$, which occurs at $n=7\,613\,644\,886$ \cite{Hur18}, and are stored as 32-bit integers beyond that point.
   This improves cache behavior for the $S_1$ lookups, most of which use the 16-bit representation.

   The second phase continues from $\nu_x$ up to $u$, where the only remaining updates come from $S_1$ and only the Mertens values are needed.
   Thus the residual array in the compressed Mertens storage can reuse the array that held $\mu$, rather than storing a second array of the same size.
   The sieve also uses larger segments in this phase, since the $M$ lookups are already far enough apart that cache misses are unavoidable and larger segments let the bucket scheduler do more useful work per pass.

   \subsection{Recovery Pass}
   \label{recoverypass}

   Let $N = \lfloor x/u\rfloor$, and write $V_k = M(x/k)$.
   After the segment passes are complete, the computed quantities are
   $$ R_k = \begin{cases}
      S(x/k,u)-S(x/(2k),u), & k\leq N/2,\\
      S(x/k,u), & k>N/2,
      \end{cases} $$
   for odd square-free $k\leq N$.
   If one only wants $M(x)$, the $R_k$ can simply be summed with their $\mu(k)$ weights, using Formula (\ref{formula2split}) applied to Theorem \ref{mainformula}.
   Instead, the implementation uses the identity (\ref{Ssumidentity}) to form a triangular system of linear equations.
   This recovers many more $V_k$, which are used in the validation checks of Section \ref{recoveredvalues}.

   For each odd square-free $k\leq N$, the corresponding linear equation is
   $$ \sum_{\substack{q\leq N/k\\q\ {\rm odd}}}V_{kq}=R_k. $$
   This system is solved through back substitution, with the indices processed in decreasing order.
   The solved value of $V_1$ is the desired value $M(x)$.

   Not every $V_k$ appearing in this system can be solved for individually.
   Appendix~\ref{recoveredappendix} gives the exact recoverability criterion.

   The following pseudocode summarizes the isolated-value computation.

\vspace{-5pt}

   \begin{algorithm}
     \SetArgSty{}
     \DontPrintSemicolon
    Sieve $\mu$ on $[1,\lfloor x/u\rfloor]$ and determine the square-free indices $k$\;
   For each of these $k$, precompute $x_k$, $\nu_{x_k}$, $\kappa_{x_k}$, and other split points\;
   \For{segments up to $\nu_x$}{
     Sieve $\mu$ on the segment and form compressed $M$\;
     Apply all $S_2$ updates meeting this segment\;
      Apply all $S_1$ updates meeting this segment\;
      Add any $\kappa_{x_k} M(\nu_{x_k})$ terms whose $\nu_{x_k}$ lies in the segment\;
    }
    \For{large segments up to $u$}{
      Sieve byte-encoded values and form compressed $M$ in place\;
      Apply $S_1$ updates only\;
     }
    Run the Mertens value recovery pass through the square-free indices\;
   \end{algorithm}

\vspace{-16pt}

   \section{Implementation Details}
   \label{implementation}

   The preceding sections give the algorithmic structure.
   This section details implementation choices that affect hot loops and hardware-specific behavior.

   \subsection{Quotient Computation}
   \label{quotientcomp}

   The main loops spend much of their time computing the quotients in the $S_1$ and $S_2$ sums.
   For native 64-bit quotients, hardware division is relatively cheap on the tested Apple Silicon machines, while on many x86 machines it is not.
   For the largest arguments, however, the computation uses 128-bit quotients, which are slower still.
   The implementation therefore has a direct method and a division-free method.
   The direct method uses the hardware division instruction.
   The division-free method is always used for 128-bit quotients.
   For 64-bit quotients, it is optional and is used only when it is faster than native division on the target machine.

   For an argument $y$ in one of the sums, the division-free method has two parts.
   It sets the split point
   $$ n_0 = \left\lceil \sqrt[3]{2y}\,\right\rceil. $$
   For $n\leq n_0$, the \emph{Quotient Cache} precomputes multiplier data for a range of denominators, replacing each quotient by a 128-bit multiply and shifts, using the invariant-division method of Granlund and Montgomery \cite{GM}.
   For $n>n_0$, the \emph{Quotient Predictor} avoids most divisions by predicting each quotient from the two preceding values.
   If $q_n=\lfloor y/n\rfloor$, then the second-order finite difference of $y/n$ is approximately $2y/n^3$, which is less than $1$ once $n>n_0$.
   The resulting prediction for $q_n$ is
   \begin{equation}
   \label{qpred}
   q_{\mathrm{est}} = 2q_{n-1}-q_{n-2}
   \end{equation}
   and is verified by multiplying $q_{\mathrm{est}}$ by $n$.
   If the product is too large, the quotient is decremented.
   If increasing the quotient still leaves the product at most $y$, the quotient is incremented until the correct value is reached.
   This verification gives the unique integer $q$ satisfying $q\,n\leq y<(q+1)\,n$ for the denominator being checked, so the predictor changes only the cost of quotient evaluation.
   In the \emph{step-$1$} case, where consecutive denominators are visited, the correction in the range $n>n_0$ is always one of $-1$, $0$, $1$, or $2$.

   The $S_1$ loop uses the step-$1$ predictor when iterating over consecutive denominators, and a \emph{step-$2$} variant when restricted to one parity.
   The $S_2$ loop visits only indices coprime to $6$.
   Thus the denominator increments alternate between $2$ and $4$.
   The predictor has a separate variant for this pattern.
   Across a step of $4$, it first makes an unverified prediction for the missing quotient, then uses that value to make the step-$2$ prediction and verifies the result by multiplication.
   This \emph{alternating step-$2/4$} variant has a wider correction window, and the implementation uses the same verify-and-correct loop rather than assuming the step-$1$ bounds.

   The Quotient Predictor always verifies its estimate by multiplication.
   Thus the choice between the direct and division-free methods changes performance, but not correctness.
   In all variants used here, correction $0$ is the favorable case.
   For the step-$1$ case analyzed in Appendix \ref{qpredappendix}, the fractional-part heuristic predicts that corrections of size $1$ decay like $x^{-1/6}$, while corrections of size $2$ decay like $x^{-1/3}$.
   Thus the exact prediction frequency tends to $1$ in the model.
   Table \ref{table:qpred} gives the corresponding measured frequencies for the step-$1$ predictor over consecutive denominators in the range $\sqrt[3]{2x}<n\leq x^{2/3}$.
   Here a correction of $-1$ denotes one decrement, while corrections of $1$ and $2$ denote increments by the indicated amount.
   The rows through $10^{12}$ are exhaustive.
   The row for $10^{25}$ uses 50 million random denominators.

   \begin{table}[!ht]
   \begin{center}
   \begin{tabular}{ | c || r | r | r | r | }
   \hline
   $x$ & correction $-1$ & correction $0$ & correction $1$ & correction $2$\\
   \hline
   $10^{10}$ & $2.24\%$ & $95.49\%$ & $2.27\%$ & $0.0031\%$\\
   $10^{11}$ & $1.54\%$ & $96.90\%$ & $1.55\%$ & $0.0014\%$\\
   $10^{12}$ & $1.06\%$ & $97.88\%$ & $1.06\%$ & $0.00067\%$\\
   $10^{25}$ & $0.0071\%$ & $99.99\%$ & $0.0072\%$ & $0$\\
   \hline
   \end{tabular}
   \end{center}
   \caption{Step-$1$ Quotient Predictor correction frequencies.}
   \label{table:qpred}
   \end{table}

   \subsection{Compressed Mertens Lookups}

   The compressed Mertens representation described in Section \ref{sieve} is central to the performance of $S_1$.
   Its purpose is to trade a slightly more complicated lookup for a substantially smaller working set and less memory traffic.
   With stride $H$, the coarse array contributes one full Mertens value per $H$ entries, while the residual array contributes one byte per entry.
   For the stride $H=256$ used in the record runs, this is slightly more than one byte per entry, compared with four bytes per entry for an uncompressed 32-bit array.
   For this stride, the signed-byte residual remains well within range.
   A heuristic of Ng \cite{NG} models $\mu$ as independent random variables with $\Pr(\mu(n)\neq 0)=6/\pi^2$ and predicts
   $$ \max_{t\leq u}|M(t+H)-M(t)|\sim \sqrt{\frac{12H}{\pi^2}\log\left(\frac{u}{H}\right)} $$
   over blocks of length $H$ up to $u$.
   This heuristic places the expected overflow far beyond the sieved ranges used here.
   Nonetheless, the code checked for residual overflow through the value of $u$ used in the $M(10^{25})$ computation.

   A lookup now requires one access into the coarse array and one access into the residual array.
   In the tested configurations, this extra lookup cost is outweighed by the smaller working set and by the larger feasible segment sizes, which give the bucket scheduler more useful work per pass.
   The balance is machine-dependent.
   Slower random access makes the $S_1$ lookups more expensive, while shifting work toward $S_2$ increases quotient work and, at the top end, use of the 128-bit path.

   \subsection{Prefetching}

   At large $x$, the compressed arrays exceed cache, and the $S_1$ access pattern is not contiguous.
   In the predictor-based $S_1$ path, an unverified step-$10$ use of the Quotient Predictor gives a cheap estimate of a future lookup position.
   The distance was chosen empirically, long enough for the prefetch to be useful in the tested loops but short enough that the estimate stays close.
   Since the value is used only as a prefetch target, it does not need to be exact.
   The implementation issues prefetch hints for both the coarse and residual arrays at that predicted location.
   When the buckets store hit positions, the scheduler similarly prefetches these upcoming locations, which are explicit in its entry stream but not visible to hardware prefetchers.
   This has no effect on correctness, but helps hide memory latency when the lookup stream reaches DRAM.

   \subsection{Parallelism and Load Balancing}

   Parallelism is organized differently in the sieve and in the $S_1$ and $S_2$ updates.
   In the sieve, each thread owns a consecutive range of sub-segments inside the current segment, together with a private circular buffer.
   The thread processes this range in order.
   This preserves the next-hit state from one sub-segment to the next, avoids lock contention, and avoids recomputing large-prime positions when the thread advances.
   For the reported computations, the sub-segment length was $L=887\,040$, chosen so that the active sub-segments across worker threads fit in cache.
   The buffer used $B=512$ buckets.
   These values satisfy $\sqrt{u}<BL$ throughout the computations and, with the $f_x$ formula for the isolated-value computation, support inputs to roughly $x=5.9\cdot 10^{26}$.

   The large-prime buckets can be stored in two ways.
   One stores only the prime, recomputing the hit position by division when the bucket is processed.
   The other also stores the hit position.
   The reported computations used the prime-only version, which reduces bucket traffic.
   On other machines, storing the hit position as well may be faster, depending on the balance between memory traffic, core count, and division cost.

   The $S_2$ update is split more finely.
   The smallest stored square-free indices have much larger $S_2$ ranges than the later ones, and assigning one index to one worker gives poor load balance.
   To balance this work, the implementation creates a list of chunks.
   A chunk records an index $i$ and a subinterval $[a,b]$ of the current sieved segment.
   Large $S_2$ tasks are split into many chunks, while small tasks remain whole.
   These chunks are then processed with dynamic scheduling.

   Inside each $S_2$ chunk, the innermost $S_2$ loop visits the residue classes coprime to $6$ in a period of $36$.
   Each $S_2$ chunk is dispatched to a routine specialized for its mode, where a mode is one of the piecewise quotient summands listed in Section \ref{ie}.
   Thus the mode-specific summand is inlined at compile time, and the main remaining branch is the check for $\mu(j)=0$.

   In the chunked $S_2$ path, several chunks may contribute to the same stored value, so the partial sums are accumulated with synchronization.
   For 64-bit entries this is done with atomic updates.
   For the much smaller set of 128-bit entries a critical section is used.
   Since the 128-bit entries are few, this does not dominate the runtime.

   The $S_1$ ranges also vary widely.
   A chunked $S_1$ variant was tested, but in both phases it was slower than dynamic scheduling over whole entries.
   The implementation therefore uses the simpler $S_1$ scheduling.

   \subsection{Backend Choices}

   The code is written in C++17 and parallelized with OpenMP.
   Most of the algorithm is shared across machines, but a few low-level choices depend on the target architecture.
   The segmented sieve has SIMD finalization kernels for ARM and x86-family instruction sets, including NEON, SVE2, SSE2, AVX2, and AVX-512 when available.
   The isolated-value computation can use either the direct or division-free quotient methods described above.
   The timings reported below were collected on Apple Silicon machines, with the computational environments summarized alongside the timing data.

   \section{Computational Results}
   \label{results}

   \subsection{Isolated Values}
   \label{isolatedresults}

   The isolated-value implementation gives the values:
   \begin{align*}
     M(10^{24}) &= 7\,189\,337\,839,\\
     M(10^{25}) &= -258\,560\,632\,948.
   \end{align*}
   The timings below use the following two machines.
   The \emph{record machine} was an M3 Ultra Mac Studio with $512$ GB of memory, while the \emph{laptop machine} was an M2 Max MacBook Pro with $32$ GB of memory.
   Both used the NEON backend and direct division in the 64-bit quotient path.
   The large sieve segments were approximately $400$ GB and $12$ GB, respectively.
   On the record machine, the computation of $M(10^{24})$ took $7.0$ days, while $M(10^{25})$ took $34.6$ days.
   The latter computation required sieving through
   $$ u=13\,694\,622\,981\,236\,974, $$
   slightly beyond the endpoint $10^{16}$ reached in \cite{Hur18}.

   The tuning parameters of Section \ref{parameterchoices} were chosen on these machines.
   Different memory systems and division costs may favor different quotient paths, split constants, and segment sizes.
   The record machine timings benefit from high memory bandwidth and fast integer division, while the noncontiguous Mertens value lookups remain sensitive to memory latency.

   Table \ref{table:isolatedtimings} gives timings on both machines.
   The per-decade ratios are slightly above the $10^{2/3}\approx 4.64$ suggested by the leading $x^{2/3}$ term.
   This reflects the slowly growing factors hidden in the $O(x^{2/3+\varepsilon})$ notation.
   These include the memory-sensitive $S_1$ Mertens value lookups, bucket scheduler overhead in the sieve, and the growing use of the slower 128-bit quotient path.
   The row at $10^{23}$ also reproduces the independently computed Helfgott-Thompson value from \cite{HT}.

   \begin{table}[ht]
   \centering
   \begin{tabular}{ | c | r || r | c || r | c | }
   \hline
   \multirow{2}{*}{$x$} & \multirow{2}{*}{$M(x)$} & \multicolumn{2}{c||}{record machine} & \multicolumn{2}{c|}{laptop machine}\\
   \cline{3-6}
    &  & \multicolumn{1}{c|}{time} & ratio & \multicolumn{1}{c|}{time} & ratio\\
   \hline
   $10^{16}$ & $-3\,195\,437$ & $2.39$ s & --- & $6.79$ s & ---\\
   $10^{17}$ & $-21\,830\,254$ & $10.82$ s & $4.53$ & $31.9$ s & $4.71$\\
   $10^{18}$ & $-46\,758\,740$ & $49.9$ s & $4.61$ & $152.2$ s & $4.77$\\
   $10^{19}$ & $899\,990\,187$ & $237.6$ s & $4.76$ & $733.9$ s & $4.82$\\
   $10^{20}$ & $461\,113\,106$ & $1150$ s & $4.84$ & $0.99$ h & $4.85$\\
   $10^{21}$ & $-3\,395\,895\,277$ & $1.52$ h & $4.76$ & $4.81$ h & $4.86$\\
   $10^{22}$ & $-2\,061\,910\,120$ & $7.26$ h & $4.76$ & $23.3$ h & $4.84$\\
   $10^{23}$ & $62\,467\,771\,689$ & $1.45$ d & $4.79$ & $4.76$ d & $4.90$\\
   $10^{24}$ & $7\,189\,337\,839$ & $7.01$ d & $4.84$ & \multicolumn{1}{c|}{---} & ---\\
   $10^{25}$ & $-258\,560\,632\,948$ & $34.6$ d & ${{}^{\phantom{*}}}4.94^*$ & \multicolumn{1}{c|}{---} & ---\\
   \hline
   \end{tabular}
   \captionsetup{skip=2pt}
   \caption{Timings for the isolated-value implementation. Each ratio compares the runtime with that at the preceding power of\,\,\nobreak$10$.}
   \label{table:isolatedtimings}
   \end{table}
   
   The $10^{25}$ runtime, marked with an asterisk in Table \ref{table:isolatedtimings}, is likely a slight overestimate of the raw computation time.
   This run included checkpointing and residual-overflow checks whose costs were not measured separately, so the reported time is an end-to-end operational figure rather than a controlled scaling measurement.
   The timing is still included because it gives useful evidence that the implementation remains stable at the largest scale reached here.

   The laptop machine ratios run higher throughout because its narrower memory bandwidth and smaller cache make the memory-bound sieve and noncontiguous Mertens value lookups grow more steeply with $x$.
   At the largest inputs, this effect is compounded by the smaller final segment that its limited memory forces.

   \subsection{Scaling and Component Balance}

   Figure \ref{fig:balance} shows the timing breakdown through $10^{24}$, with the sieve, $S_1$, and $S_2$ totals aggregated across the 16- and 32-bit variants of the first phase and the larger-segment second phase.
   For large enough inputs, the tuned choices of $f_x$ and $c_x$ keep the relative timings of the three main components within a narrow band.
   As $x$ grows, the bucket scheduler becomes more important in the sieve, while a larger part of the summation moves onto the 128-bit path.
   The high-decade values of $f_x$ were chosen mainly to balance these two pressures.
   
   \newpage

   \begin{figure}[ht]
   \centering
   \includegraphics[width=.85\textwidth]{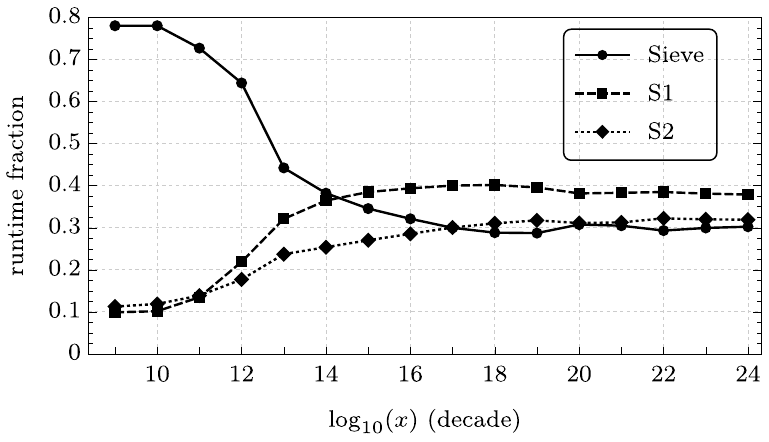}
   \captionsetup{skip=2pt}
   \caption{Fraction of runtime spent in each major subroutine.}
   \label{fig:balance}
   \end{figure}
   
   Figure \ref{fig:paramtune} shows the tuning sweep for the factor $f_x$ in the sieve bound $u$, holding $c_x$ fixed at $1.5$.
   For decades $10^9$ through $10^{22}$, the point marks the fastest tested value of $f_x$, and the band gives the $f_x$ values whose runtimes are within $2.5\%$ of the fastest.
   The final three decades computed are not included in this sweep.
   Their parameter choices came from smaller local timing tests.

   \begin{figure}[ht]
   \centering
   \includegraphics[width=.85\textwidth]{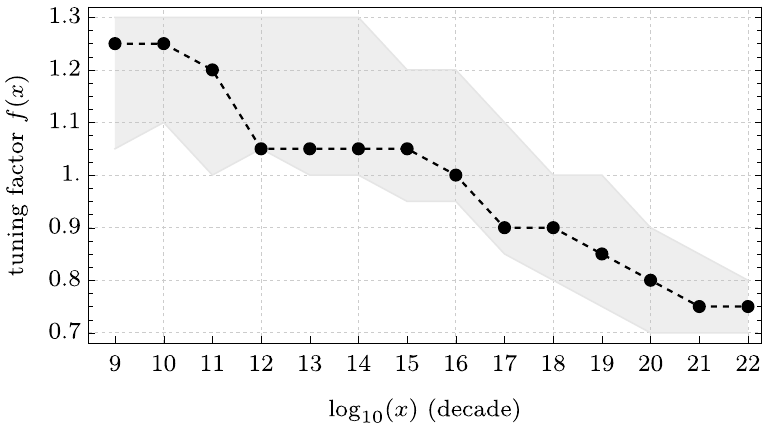}
   \captionsetup{skip=2pt}
   \caption{Tuning sweep for the factor $f_x$ used in the sieve bound. Points mark the fastest tested value in each decade, and the band gives values within $2.5\%$ of the fastest runtime.}
   \label{fig:paramtune}
   \end{figure}
   \FloatBarrier

   Finally, Figure \ref{fig:scaling} gives decade-to-decade timing ratios.
   The dashed line is the pure $x^{2/3}$ baseline, $10^{2/3}\approx 4.64$.
   Departures reflect overheads, parameter changes, memory effects, and the onset of the bucket scheduler visible in the sieve panel at $10^{20}$.
   The sieve ratios fall after this transition as the tuned values of $f_x$ decrease through $10^{22}$, then move back up once $f_x$ is held fixed.

   \begin{figure}[ht]
   \centering
   \includegraphics[width=1\textwidth]{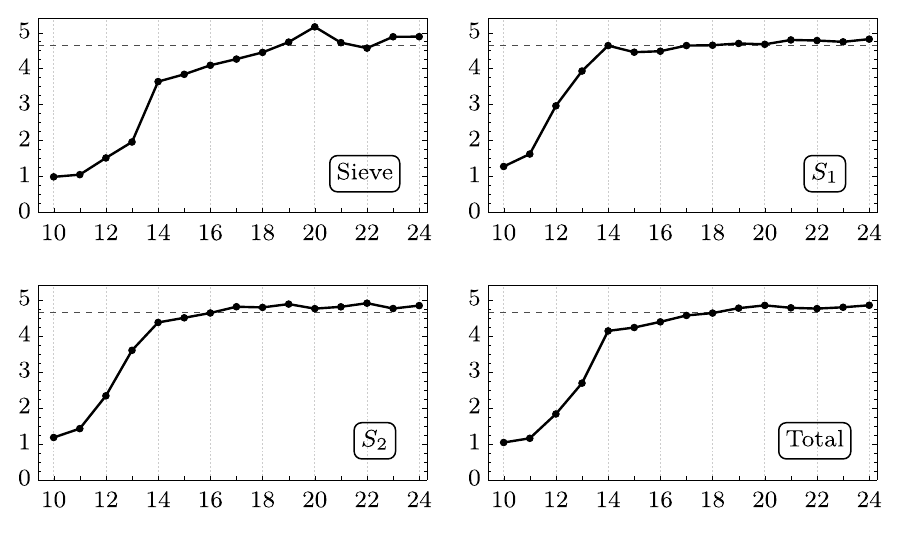}
   \captionsetup{skip=2pt}
   \caption{Ratio of runtimes at successive powers of $10$. In each panel, the horizontal axis is the decade $k$, and the vertical axis is the ratio $T(10^k)/T(10^{k-1})$. The dashed horizontal line is $10^{2/3}\approx 4.64$, the ratio predicted by pure $x^{2/3}$ scaling.}
   \label{fig:scaling}
   \end{figure}
   \FloatBarrier
   
   \subsection{Standalone Sieve}

   The sieve is also a standalone tool, so it can be timed independently of the isolated-value computation.
   The standalone run considered here follows \cite{Hur18} by computing $\mu(n)$ and $M(n)$ for all $n\leq 10^{16}$, without any isolated-value updates.
   In \cite{Hur18}, this took about $7.5$ months, with the first segment of length $10^{14}$ taking about one day and the final such segment taking about $2.8$ days.
   On the record machine, the sieve described above computes through $10^{16}$ in $7.39$ days, with the first segment of length $10^{14}$ taking $1.65$ hours and the final segment ending at $10^{16}$ taking $1.84$ hours.
   Thus the later segments are only moderately slower than the first, rather than becoming several times slower.

   The above comparison with \cite{Hur18} mixes hardware and implementation changes.
   To separate these effects, the earlier sieve was also run on the record machine.
   Table \ref{table:sievescaling} compares the three segments of length $10^{14}$ singled out in \cite{Hur18}.
   In the following tables, endpoints are scaled by $10^{14}$, e.g., $[99,100]$ denotes $[9.9\cdot 10^{15},10^{16}]$.

   \begin{table}[ht]
   \setlength{\tabcolsep}{4pt}
   \begin{center}
   \begin{tabular}{| c || c | c | c |}
   \hline
   segment & sieve \cite{Hur18} & this work & speedup\\
   \hline
   $\,\hphantom{9}[0,1]\hphantom{00}$ & $\hphantom{0}4.12$ h & $1.48$ h & $\,2.79 \smalltimes$ \\
   $\,\hphantom{9}[1,2]\hphantom{00}$ & $\hphantom{0}4.83$ h & $1.57$ h & $\,3.08 \smalltimes$ \\
   $\,[99,100]$ & $11.44$ h & $1.84$ h & $\,6.23 \smalltimes$ \\
   \hline
   \end{tabular}
   \end{center}
   \caption{Same-machine comparison with the segmented sieve from \cite{Hur18} on selected standalone sieve segments. All timings are on the record machine.}
   \label{table:sievescaling}
   \end{table}

   Over large sieve ranges, the bucket scheduler is the main source of speedup.
   To measure its effect directly, the same segments were run on the record machine with this feature both disabled and enabled.
   
   \vspace{-5 pt}

   \begin{table}[ht]
   \setlength{\tabcolsep}{4pt}
   \begin{center}
   \begin{tabular}{| c || c | c | c |}
   \hline
   segment & no buckets & buckets & speedup \\
   \hline
   $\,\hphantom{9}[0,1]\hphantom{00}$ & $2.30$ h & $1.48$ h & $\,1.56 \smalltimes$ \\
   $\,\hphantom{9}[1,2]\hphantom{00}$ & $2.98$ h & $1.57$ h & $\,1.90 \smalltimes$ \\
   $\,[99,100]$ & $8.16$ h & $1.84$ h & $\,4.44 \smalltimes$ \\
   \hline
   \end{tabular}
   \end{center}
   \caption{Effect of the large-prime bucket scheduler on the same standalone sieve segments. All timings are on the record machine.}
   \label{table:sievebuckets}
   \end{table}
   
   \vspace{-20 pt}

   Tables \ref{table:sievescaling} and \ref{table:sievebuckets} separate two sources of improvement.
   Even with the bucket scheduler disabled, the present sieve is faster than the earlier same-machine implementation on all three segments.
   Thus, though the bucket scheduler gives the largest performance increase on segments with large endpoints, the other sieve and implementation changes from Sections \ref{sieve} and \ref{implementation} are substantial.

   Table \ref{table:sieveprofile} gives cumulative timings for the standalone sieve at powers of $10$.
   Each ratio compares the runtime with that at the preceding limit.
   The larger ratios near $10^{13}$ and $10^{14}$ occur where the large-prime bucket scheduler begins to contribute substantially.
   After this transition, the ratios start to settle back toward the near-linear behavior expected for segmented sieving, with the decade ratios approaching a value just above 10.
   
   \vspace{-5 pt}

   \begin{table}[ht]
   \centering
   \begin{adjustbox}{max width=\textwidth}
   \begin{tabular}{ | c || r | c || r | c || c | }
   \hline
   \multirow{2}{*}{limit} & \multicolumn{2}{c||}{$\mu$ sieve} & \multicolumn{2}{c||}{$M$ sieve} & \multirow{2}{*}{segment}\\
   \cline{2-5}
    &  \multicolumn{1}{c|}{time} & ratio & \multicolumn{1}{c|}{time} & ratio &\\
   \hline
   $10^{9\hphantom{0}}$ & $0.038$ s & --- & $0.067$ s & --- & $4 \cdot 10^{7\hphantom{0}}$\\
   $10^{10}$ & $0.25$ s & $\hphantom{0}6.74$ & $0.33$ s & $\hphantom{0}4.93$  & $4 \cdot 10^{8\hphantom{0}}$\\
   $10^{11}$ & $2.66$ s & $10.47$ & $3.17$ s & $\hphantom{0}9.61$  & $4 \cdot 10^{9\hphantom{0}}$\\
   $10^{12}$ & $31.14$ s & $11.71$ & $35.31$ s & $11.14$  & $4 \cdot 10^{10}$\\
   $10^{13}$ & $433.7$ s & $13.93$ & $469.3$ s & $13.29$  & $4 \cdot 10^{11}$\\
   $10^{14}$ & $1.38$ h & $11.45$ & $1.48$ h & $11.35$ & $4 \cdot 10^{11}$\\
   $10^{15}$ & $15.34$ h & $11.12$ & $16.27$ h & $10.99$ & $4 \cdot 10^{11}$\\
   $10^{16}$ & $7.00$ d & $10.94$ & $7.39$ d & $10.90$ & $4 \cdot 10^{11}$\\
   \hline
   \end{tabular}
   \end{adjustbox}
   \captionsetup{skip=2pt}
   \caption{Timings for the standalone sieve on the record machine through the listed endpoint.
   The segment column gives the sieve segment length used for that endpoint.}
   \label{table:sieveprofile}
   \end{table}
   
      \vspace{-20 pt}

   \subsection{An Additional Isolated Value}

   As one further benchmark, the implementation was also run at Avogadro's number, a natural large input near the scale of the largest computations.
   In the usual notation, $N_A = 6.02214076\times 10^{23}$.
   Since the 2019 redefinition of the SI base units, this value is exact, namely
   $602\,214\,076\cdot 10^{15}$.
   This also exercises the implementation at a large input that is not a power of $10$.
   On the record machine, the computation took $4.9$ days, giving
   $$ M(N_A) = 104\,076\,878\,054. $$
   
   \vspace{20 pt}

   \section{Validation}
   \label{validation}

   The scale and complexity of the implementation make validation especially important, and so the checks below are meant to be thorough and complementary.
   Some test the segmented sieve, some test the isolated-value summation, and some test only coarse features such as parity or size.
   Taken together, they give evidence against several failure modes, without amounting to full independent reruns.

   \subsection{Checks Against Previous Data}

   The segmented Mertens sieve was tested against the data computed in \cite{Hur18}.
   That computation stored values of $M(x)$ for all multiples of $10^8$ up to $10^{16}$, as well as extrema and zeros.
   This data provides a large set of test points for both the standalone sieve and the isolated-value routine.

   The isolated-value code was run on known powers of $10$ and powers of $2$.
   In particular, it reproduces the values reported in \cite{DR}, \cite{EK}, \cite{Hur18}, and \cite{HT} in the range where they are available.
   The agreement at $10^{23}$ is a particularly strong check, since \cite{HT} computed this value independently using a different algorithm.
   In fact, the optimized Helfgott-Thompson implementation of Section \ref{htcompare} provides an independent check for smaller inputs where both programs can be run in reasonable time.

   \subsection{M\"obius Sieve Identity}

   The segmented M\"obius sieve was tested against the identity
   $$ \sum_{n\leq x}\mu(n)\left\lfloor \frac{x}{n}\right\rfloor = 1. $$
   Here, every value of $\mu(n)$ up to $x$ contributes, making it a natural test.
   It was verified for powers of $10$ from $10^9$ through $10^{16}$ by sieving through $x$ and summing the left-hand side in parallel over the generated segments.
   This is a global consistency check rather than a proof that each $\mu$ is correct.

   \subsection{Verification Runs}

   The new values at $x=10^{24}$ and $x=10^{25}$ were checked by recomputing nearby arguments on the record machine using different tuning parameters.
   For each of these two inputs, an independently chosen random integer $r$ near $10^7$ was used to compute $M(x-r)$.
   The value of $M(x)$ was then recovered by adding the remaining $r$ M\"obius values, computed separately in Mathematica.
   For both values of $x$, these verification runs used $f_x=0.7$ and $c_x=1.6$.
   The results agreed with the values reported here.

   Although the isolated-value computation is rerun at $x-r$, the check is not as independent as it first appears.
   Since $r$ is small relative to $x$, many of the quotients appearing in the computations of $M(x)$ and $M(x-r)$ are identical.
   Nevertheless, the test perturbs the endpoint and the summation boundaries, uses a separately computed terminal interval, and checks that the agreement is not an artifact of evaluating the algorithm only at powers of $10$.

\subsection{Parity Checks}

   The parity of $M(x)$ coincides with that of the square-free counting function $Q(x)=\sum_{n\le x} |\mu(n)|$.
   The well-known identity
   $$ Q(x)=\sum_{d\leq \sqrt{x}}\mu(d)\left\lfloor\frac{x}{d^2}\right\rfloor $$
   then gives an independent way to compute $Q(x)$ in $O(x^{1/2+\varepsilon})$ time.

   This check was run for powers of $10$ through $10^{25}$ using a separate plain segmented M\"obius sieve, without the enhancements described in Section \ref{sieve}.
   The formula above was summed in parallel modulo $2$.
   In each tested decade, the parity of $Q(10^n)$ agreed with that of the reported $M(10^n)$ values.

   \subsection{Analytic Sanity Checks}

   The values at $10^{24}$ and $10^{25}$ were compared with estimates from the truncated explicit formula, as in \cite{Hur18}.
   Let
   $$ E_N(x)=2\sqrt{x}\sum_{i=1}^N
      \frac{\cos(\gamma_i\log x+\psi_i)}
           {|\rho_i\zeta'(\rho_i)|} $$
   with $N=5\,000\,000$, where $\rho_i$ runs over the nontrivial zeros of $\zeta$ with positive imaginary part,
   $\gamma_i=\operatorname{Im}(\rho_i)$, and
   $\psi_i=-\arg(\rho_i\zeta'(\rho_i))$.
   This estimate tends to give the sign, order of magnitude, and a rough indication of the leading digits, except when $M(x)$ happens to be small.
   To handle such cases, the relevant diagnostic was the normalized residual
   $$ \frac{M(x)-E_N(x)}{\sqrt{x}}. $$
   The following table gives these residuals for the largest powers of $10$ considered here.
   The values at $10^{24}$ and $10^{25}$ are comparable in size to those observed at the previously known inputs.

   \begin{center}
   \begin{tabular}{ | c | r || c | r | }
   \hline
   $x$ & \multicolumn{1}{c||}{residual} & $x$ & \multicolumn{1}{c|}{residual}\\
   \hline
   $10^{16}$ & $-0.000293075$ & $10^{21}$ & $-0.000014991$\\
   $10^{17}$ & $\phantom{-}0.000266250$ & $10^{22}$ & $-0.000289840$\\
   $10^{18}$ & $\phantom{-}0.000420964$ & $10^{23}$ & $-0.000254179$\\
   $10^{19}$ & $-0.000121069$ & $10^{24}$ & $\phantom{-}0.000387819$\\
   $10^{20}$ & $-0.000054745$ & $10^{25}$ & $-0.000281637$\\
   \hline
   \end{tabular}
   \end{center}

   The estimates $E_N(x)$ were also examined as the number of zeros increased to $5\,000\,000$.
   Their behavior was consistent with stabilization toward the reported estimates, rather than an isolated agreement at the final truncation.
   This is not a proof of correctness, but it gives an independent check from different code and mathematics.

   \subsection{Internal Checks}

   The preceding checks compare the output with separate computations.
   A different class of failure is internal, such as a compressed residual overflowing, an accumulator exceeding its intended range, or a specialized arithmetic path being used outside the range for which it was checked.

   The residual array in the compressed Mertens representation was checked for overflow through the value of $u$ used at $10^{25}$.
   This is well below the range where the Ng heuristic discussed in Section \ref{implementation} predicts overflow.

   Crude absolute-value estimates for the running accumulations in $S_1$ and $S_2$, and for the stored combinations formed from these sums, remain below $2^{63}$ throughout the ranges handled by the 64-bit accumulation paths.
   In ranges where both 64-bit and 128-bit quotient paths apply, the two paths were compared on various inputs.
   Additional implementation bounds, including those that would need to be revisited for larger inputs, are documented with the released code.

   \subsection{Recovered Intermediate Values}
   \label{recoveredvalues}

   The recovery pass described in Section \ref{recoverypass} produces a large family of intermediate values that can be checked independently.
   This is the same kind of check used in earlier implementations of the classical method, but the inclusion-exclusion reductions change which values remain recoverable.
   If $N=\lfloor x/u\rfloor$, then the Mertens value recovery pass recovers $M(x/k)$ for any odd square-free $k\leq N$ satisfying
   $$ P^-(k)k>N, $$
   where $P^-(k)$ is the least prime factor of $k$ and $P^-(1)=\infty$.
   Appendix \ref{recoveredappendix} proves this criterion and derives the corresponding density of approximately $0.351537$.

   For the computation at $10^{25}$, the sieve bound gives $N=730\,213\,603$.
   In addition to the main value $M(x)$ coming from $k=1$, the recovery pass gives $256\,697\,925$ intermediate values.
   Their arguments $x/k$ range from approximately $1.37\cdot 10^{16}$ to $3.70\cdot 10^{20}$.
   Thus the surviving family does not contain the largest arguments of the form $M(x/k)$, apart from $M(x)$ itself.
   A stratified random sample of $900$ of these values was recomputed using a separate, more direct implementation.
   The sample contained $400$ arguments between $10^{16}$ and $10^{17}$, $200$ between $10^{17}$ and $10^{18}$, $150$ between $10^{18}$ and $10^{19}$, $100$ between $10^{19}$ and $10^{20}$, and $50$ above $10^{20}$.
   These direct checks took approximately $9$ days in total, and all values agreed.

   \subsection{Direct Checks on the Largest Evaluations}

   The preceding checks leave the largest evaluations of $S(y,u)$ least directly tested.
   These occur for small outer indices $k$, where $y=x/k$ lies above the range reached by the recovered intermediate values.
   A single evaluation of $S(y,u)$ costs only $O(\sqrt{y})$, so these values can be checked directly even when a full recomputation of $M(x)$ is not practical.

   For $S_2$, the optimized computation was compared with a direct evaluation from the definition.
   This check was run on values from the top 128-bit range, from an intermediate range, and from a range where the quotients are 64-bit.
   The comparison follows the two cases from Section \ref{ie}, namely $S_2(x/k)-S_2(x/(2k))$ in the combined range and $S_2(x/k)$ in the single range.
   The direct evaluation recomputed these quantities from the defining sum
   $$ S_2(y)=\sum_{n\leq \nu_y}\left\lfloor \frac{y}{n}\right\rfloor\mu(n) $$
   using a separate plain segmented M\"obius sieve, ordinary 128-bit division, and 128-bit accumulators.
   The reference computation does not use the quotient predictor, the optimizations described in Section \ref{sieve}, the coprime-to-$6$ reduction, or the tables of piecewise summands in Section \ref{ie}.
   Thus agreement gives a direct check of the tested $S_2$ values, including the reductions, merged breakpoints, quotient arithmetic, and accumulation.
   All tested values agreed.

   For $S_1$, an analogous recomputation would require an independent table of Mertens values up to $u$.
   Instead, the test replaces the Mertens lookup by a deterministic surrogate
   $$ M_{\mathrm{proxy}}(t)=A(t)\sqrt{t}\cos(\gamma_1\log t),
      \qquad
      A(t)=0.5h(\lfloor 100\log_{10}t\rfloor), $$
   where $h\colon\mathbb{Z}\to[0,1)$ is a hash function and $\gamma_1$ is the ordinate of the first nontrivial zero of $\zeta$.
   The surrogate is not meant to approximate $M(t)$ closely.
   Its role is to give the $S_1$ loop deterministic lookup values that vary with the queried argument and have roughly the same size as the true Mertens values.
   The hashed amplitude adds mild irregularity, while the cosine term gives the values a log-scale oscillation suggested by the explicit formula.
   The amplitude bound $0.5$ was chosen from the size of $M(t)/\sqrt{t}$ observed in \cite{Hur18} for $t\leq 10^{16}$.

   The same input ranges were then used for the $S_1$ check.
   The optimized $S_1$ routine was run with its Mertens lookup replaced by $M_{\mathrm{proxy}}$.
   A reference implementation evaluated the same reduced sum (\ref{s1parity}) by direct division, recomputing the split points by exact integer square root.
   The two implementations agreed for every tested value.
   This does not independently verify the true $S_1$ values.
   It checks the quotient streams, parity-$2$ indexing, split points, and accumulation at the scale of the real computation.
   The Mertens inputs and their compressed storage are covered by the preceding checks.

   \subsection{Scope of the Validation}

   The checks above do not test every part of the computation with the same degree of independence.
   The sieve identities and comparisons with previous data test the generated Mertens and M\"obius values over long ranges.
   The parity check reaches the final values through $10^{25}$, but only modulo $2$.
   The analytic check tests the size and leading behavior, but not the lower digits.
   The nearby-input runs perturb the endpoint and parameters, although they still overlap substantially with the original computation.
   The recovered intermediate values test many smaller arguments against a separate implementation.
   The direct checks on the largest evaluations narrow the remaining gap by testing the largest $S_2$ values on real data and the corresponding $S_1$ summation machinery with a deterministic surrogate.

   The least directly tested piece is therefore the full $S_1$ value for the largest evaluations of $S(y,u)$.
   Its Mertens inputs and summation machinery are checked separately, but an independent recomputation of the summed values would require a separate Mertens table up to $u$.
   The loops and bookkeeping that combine the many $S(x/k,u)$ contributions are exercised end to end, but there is no separate implementation of this step at the scale of the largest computations.
   A fully independent computation of $M(10^{24})$ or $M(10^{25})$ by unrelated code or hardware would still be the strongest check.

   The evidence is therefore cumulative rather than resting on one check.
   The tests above target different failure modes, from sieve identities and range checks to endpoint perturbations, recovered values, parity checks, analytic estimates, and direct checks on the largest evaluations.
   
   \section{Practical Comparison with Helfgott-Thompson}
   \label{htcompare}

   The Helfgott-Thompson algorithm is an asymptotically superior alternative to the classical isolated-value approach.
   Its expected word operation count is
   $$ O\left(x^{3/5}(\log x)^{3/5}(\log\log x)^{2/5}\right), $$
   up to lower-order terms.
   The authors used it to compute powers of $10$ through $10^{23}$ \cite{HT}, making it not only a theoretical improvement but a demonstrated practical method.
   It is therefore natural to examine it alongside the classical implementation developed above, especially when the aim is to understand which method to prefer for isolated values of $M(x)$ as $x$ and the hardware vary.
   The newer Hirsch-Kessler-Mendlovic framework has a still better exponent, but its practical behavior remains open.

   The comparison below uses a separate optimized Helfgott-Thompson implementation based on the authors' code.
   Its mathematical decomposition is unchanged.
   The implementation changes reduce repeated divisions, improve the large-non-free factorization sieve, expose more parallel work in the large-free phase, reuse scratch buffers, and simplify correction terms in the large-free double-sum routine.
   The most important local simplifications occur in \textsc{SumByLin}, where many correction intervals are empty or full, or cancel exactly as $\operatorname{sumInter}(J)-\operatorname{sumInter}(J)$.
   Appendix \ref{htappendix} gives the implementation details, with further notes included in the code release described in Section \ref{codeavailability}.

   The Helfgott-Thompson implementation also has tunable constants.
   In the notation of \cite{HT}, the split parameter $v$ separates the large-free and large-non-free parts, while $C$ and $D$ are constants in the large-free-variable routine.
   The original column in Table \ref{table:htspeedup} uses the default settings in the authors' code, namely $v=\lfloor x^{2/5}\rfloor/3$, $C=10$, and $D=8$.
   The optimized column uses $v=\lfloor x^{2/5}\rfloor/2$, $C=42$, and $D=16$, chosen from timing sweeps on the record machine.
   The denominator in $v$ is the most sensitive constant, but at higher decades the timing curve was fairly flat around $2$, so later runs kept the same choice without their own sweep.

   \begin{table}[ht]
   \centering
   \setlength{\tabcolsep}{5pt}
   \begin{tabular}{ | c || c | c | c | }
   \hline
   $x$ & orig. HT & opt. HT & speedup\\
   \hline
   $10^{16}$ & $43.74$ s & $12.74$ s & $3.4 \smalltimes$\\
   $10^{17}$ & $242.9$ s & $59.68$ s & $4.1 \smalltimes$\\
   $10^{18}$ & $\hphantom{.}1159$ s & $276.7$ s & $4.2 \smalltimes$\\
   $10^{19}$ & $\hphantom{0}1.55$ h & $\hphantom{.}1240$ s & $4.5 \smalltimes$\\
   $10^{20}$ & $\hphantom{0}6.90$ h & $\hphantom{0}1.54$ h & $4.5 \smalltimes$\\
   $10^{21}$ & $\hphantom{0}30.4$ h & $\hphantom{0}7.05$ h & $4.3 \smalltimes$\\
   $10^{22}$ & $\hphantom{0}5.63$ d & $32.82$ h & $4.1 \smalltimes$\\
   \hline
   \end{tabular}
   \captionsetup{skip=2pt}
   \caption{Speedup of the optimized Helfgott-Thompson implementation over the authors' implementation on the record machine.}
   \label{table:htspeedup}
   \end{table}

   In the rest of this section, the implementation of Section \ref{results} is called the classical implementation, and the optimized Helfgott-Thompson code is abbreviated as optimized HT.
   Table \ref{table:htselection} compares the two implementations on the record machine.
   The $10^{24}$ and $10^{25}$ runtimes in the opt.\ HT column are projections.

   \begin{table}[!ht]
   \centering
   \setlength{\tabcolsep}{4pt}
   \begin{tabular}{ | c || c | c | c | }
   \hline
   $x$ & opt. HT & classical & ratio\\
   \hline
   $10^{16}$ & $12.74$ s & $\hphantom{0}2.39$ s & $5.3$\\
   $10^{17}$ & $59.68$ s & $10.82$ s & $5.5$\\
   $10^{18}$ & $276.7$ s & $\hphantom{0}49.9$ s & $5.5$\\
   $10^{19}$ & $\hphantom{.}1240$ s & $237.6$ s & $5.2$\\
   $10^{20}$ & $\hphantom{0}1.54$ h & $\hphantom{.}1150$ s & $4.8$\\
   $10^{21}$ & $\hphantom{0}7.05$ h & $\hphantom{0}1.52$ h & $4.6$\\
   $10^{22}$ & $32.82$ h & $\hphantom{0}7.26$ h & $4.5$\\
   $10^{23}$ & $\hphantom{0}6.21$ d & $\hphantom{0}1.45$ d & $4.3$\\
   $10^{24}$ & ${{}^{\phantom{*}}}\hphantom{00.}28$ d$^*$ & $\hphantom{0}7.01$ d & $4.0$\\
   $10^{25}$ & ${{}^{\phantom{*}}}\hphantom{0.}127$ d$^*$  & $\hphantom{0}34.6$ d & $3.7$\\
   \hline
   \end{tabular}
   \captionsetup{skip=2pt}
   \caption{Comparison of optimized HT with the classical implementation on the record machine. Asterisks mark projected runtimes.}
   \label{table:htselection}
   \end{table}

   \FloatBarrier

   Optimized HT avoids overflow restriction noted in \cite{HT} at $10^{23}$ by reducing the relevant \textsc{SumByLin} residue before the final multiplication by $n_\circ$.
   For larger inputs, guarded $256$-bit fallbacks are in place and only used for the necessary products and divisions in \textsc{SumByLin}.
   Additionally, the large-free schedule caps its working range to keep memory use feasible, which throttles the available parallel work near the highest tested decades.
   These changes should make inputs beyond $10^{23}$ attainable without further modification, though no full runs were executed at this scale.

   The optimized HT decade ratios move in the expected direction, but on the record machine they remain far from the leading $10^{3/5}$ ratio.
   Profiling shows that over the measured decades, an increasing share of the runtime falls in the large-free double-sum routine, whose \textsc{SumByLin} correction path is particularly latency-sensitive on the record machine.
   The correction-interval simplifications reduce the constant, but they do not remove the growing correction width.
   On other machines, such as the x86 machine of Appendix \ref{gpuvariant}, the decade ratios can settle faster toward $10^{3/5}$, as seen in the CPU columns of Table \ref{table:gpuht}.

   As shown in Table \ref{table:htselection}, on the record machine, the classical implementation is the more efficient choice at every input.
   The gap narrows over decades as the asymptotics suggest, but it is still about a factor of $3.7$ at the projected $10^{25}$ runtimes.
   This should be interpreted as a machine-specific result since the two implementations stress different parts of the hardware.
   The classical implementation relies more heavily on the bucket scheduler in long sieve passes and large contiguous memory traffic, while optimized HT shifts more work into latency-sensitive arithmetic in the large-free double sums.

   Memory also affects the choice between methods.
   The classical implementation stores several values for each square-free outer index, giving a floor of roughly $38$ GB at $10^{25}$ that segment tuning does not reduce.
   This floor is the $O(x^{1/3}(\log\log x)^{2/3})$ space term of \cite{DR}, while the sieve segments themselves are elastic, as the $32$ GB laptop column of Table \ref{table:isolatedtimings} shows through $10^{23}$.
   The Helfgott-Thompson algorithm has a smaller space complexity of $O(x^{3/10}(\log x)^{13/10})$, and a single factorization window of length about $\sqrt{x/v}$ needs only $1.4$ GB at $10^{25}$.
   The optimized implementation instead spends memory on parallel width, running many such windows at once, with the associated arrays reaching $182$ GB at $10^{22}$ and capped at $400$ GB in these runs.
   Thus the asymptotic space advantage belongs to Helfgott-Thompson, while both tuned implementations fill large machines.

   The broader lesson is that the classical and Helfgott-Thompson algorithms are both practical and complementary tools.
   They are asymptotically distinct, independently coded, and fast enough to give serious cross-checks on large values.
   Together they give a genuine algorithmic choice as $x$ and the hardware vary.
   
   \section{Code Availability}
   \label{codeavailability}

   The code is available at \url{https://github.com/greg-hurst/mertens}.
   It contains the standalone segmented M\"obius and Mertens sieve, the classical isolated-value code, the optimized Helfgott-Thompson code, and the experimental GPU variants of Appendix \ref{gpuvariant}.
   It also includes build instructions, command-line examples, and implementation notes.
   This release is intended to make the main implementations reusable and inspectable, rather than to archive every one-off validation or run-management script used during the computations.

   \section{Extensions and Concluding Remarks}
   \label{conclusion}

   This paper has described a segmented M\"obius and Mertens sieve used as a standalone tool, as the basis for the classical $O(x^{2/3+\varepsilon})$ isolated-value implementation, and as part of the optimized Helfgott-Thompson implementation.
   The largest isolated values reported are $M(10^{24}) = 7\,189\,337\,839$ and $M(10^{25}) = -258\,560\,632\,948$, supported by the validation evidence of Section \ref{validation}.

   \subsection{Range Extensions and Implementation Limits}

   The classical isolated-value code should already be able to compute $M(10^{26})$ without modification.
   The optimized Helfgott-Thompson code is also expected to reach that input with the $256$-bit fallbacks mentioned in Section \ref{htcompare}, although this is less tested since it has only been run up to $10^{23}$.
   The code release also contains a more detailed range analysis in \texttt{MertensHurst/INPUT\_BOUNDS.md}, including which implementation constants would need to change for still larger inputs.

   One limitation of the present parallel bucket scheduler is that bucket state is carried forward only through the sub-segments belonging to one large sieve segment.
   At the next large segment, the scheduler is rebuilt rather than continued from a global state shared by all threads.
   This was not costly here, because the chosen segment sizes made the number of rebuilds small relative to the sieve range, but it could matter for much larger sieve ranges or on machines that force substantially smaller segments.
   A future version could carry some bucket state across large segments, provided this can be done without adding too much synchronization between threads.

   \subsection{Alternatives Considered}

   Several natural alternatives were considered but not used.
   They remain possible future directions, but none fit the practical balance of the present implementations.

   Helfgott's low-space sieve of Eratosthenes \cite{HelSieve} is a natural candidate here, since it addresses the same large-divisor issue that motivates bucket scheduling.
   When the sieving segment is short relative to the large primes, most of those primes do not contribute.
   Its local construction could therefore be useful on machines with much smaller feasible segments, or in an implementation that keeps explicit factorization data.
   In the present implementation, however, the byte-logarithm encoding, bucket scheduler, and parallel layout are tightly integrated.
   Using Helfgott's sieve in their place would therefore be a substantially different low-level design rather than a drop-in replacement for the bucket scheduler.

   Another possibility for computing $\mu(n)$ is to stop the ordinary prime sieve below $\sqrt{n}$.
   Sieving by primes up to $n^{1/3}$ leaves a cofactor with at most two prime factors, and avoids most of the large-prime work that the bucket scheduler was introduced to remove.
   The remaining large-prime information is then concentrated in this cofactor.
   In a sieve that stores explicit cofactors this is natural, since one can test whether the cofactor is prime and update $\mu(n)$ accordingly.
   The compact byte-logarithm state used here does not retain that cofactor.
   Thus this approach would either enlarge the sieve state substantially, or give back much of the large-prime work it was meant to remove.

   A demand-driven Mertens table was tested. While computing the prefix sum on a segment, one can store only those $M$-values later queried by $S_1$.
   This gives more cache-friendly storage, but makes each lookup less direct than $M(q)=C[q/H]+R[q]$ inside the hottest noncontiguous access loop.
   On the machines tested here, this was not a favorable tradeoff, though the balance may differ on hardware with a different ratio of memory latency to indexing cost.

   One can also sieve only up to some $u^*<u$ and recover the missing larger Mertens values recursively.
   This loses the shared segment structure that makes the direct computation efficient.
   The $S_1$ ranges require many large values, and recursive isolated-value calls repeat setup work that the full segmented sieve performs once.

   Finally, GPU variants of both isolated-value implementations were tested.
   On the machines here, the GPU is not faster than the optimized CPU code at the offloaded arithmetic itself, so the offload pays only while its work is hidden behind independent CPU work.
   This gives gains at small inputs, most notably on the x86 machine, but does not improve the record-scale path.
   The details are given in Appendix \ref{gpuvariant}.

   \subsection{Outlook}

   The computations here show that exponent improvements are only one part of the practical problem.
   The classical isolated-value method still has considerable range when its summation identities, sieving, storage layout, and parallel scheduling are designed together.
   A natural next benchmark would be an optimized implementation of the Hirsch-Kessler-Mendlovic framework, with its $O(x^{8/15+\varepsilon})$ time and $O(x^{1/3+\varepsilon})$ space tradeoff.
   The eventual choice among this framework, Helfgott-Thompson, and future implementations depends on the input range, the implementation constants, and the machine's time and memory costs in the dominant subroutines.

   \section*{Acknowledgments}

   I am grateful to my wife, Aman, for her patience, encouragement, and support throughout this work.
   Her background in chemistry provided the motivation for the Avogadro's number computation, and our own chemistry made it especially fitting.
   I also thank Harald Andr\'es Helfgott and Lola Thompson for their work on the algorithm in \cite{HT}, for making their implementation available, and for helpful correspondence.
   
   \appendix

   \section{Quotient-Predictor Correction Frequencies}
   \label{qpredappendix}

   This appendix records the heuristic supporting the asymptotic correction claims in Section \ref{quotientcomp}.
   Only the step-$1$ predictor is treated here.
   The step-$2$ and alternating step-$2/4$ predictors have different correction sets, but the same verification step makes their use exact.
   Their correction frequencies are expected to have the same qualitative behavior.

   Let $q_n=\lfloor x/n\rfloor$, and let
   $$ c_n=q_n-2q_{n-1}+q_{n-2} $$
   be the correction to the linear prediction of $q_n$ from the two previous quotients.
   The corrections are counted on
   $$ I_x=\{n:\sqrt[3]{2x} < n \leq x^{2/3}\},\qquad N_x=|I_x|. $$
   Here $\{t\}$ denotes the fractional part, and $\|t\|=\min(\{t\},1-\{t\})$ denotes the distance from $t$ to the nearest integer.

   \begin{proposition}
   Let $C_j(x)=|\{n\in I_x:c_n=j\}|$.
   The fractional parts occurring below are modeled as locally uniform.
   Then
   $$ C_{\pm1}(x)\approx Ax^{1/2}\pm Dx^{1/3},\qquad
      C_2(x)\approx Bx^{1/3},\qquad
      C_0(x)\approx N_x-2Ax^{1/2}-Bx^{1/3}, $$
   where
   $$ A=\frac{4-\sqrt2}{2\pi}\zeta\!\left(\frac32\right),\qquad
      B=\frac{1}{5\cdot2^{5/3}},\qquad
      D=\frac{2^{1/3}}{5}. $$
   In particular, since $N_x\sim x^{2/3}$, corrections of size $1$ have frequency on the order of $x^{-1/6}$ and corrections of size $2$ have frequency on the order of $x^{-1/3}$.
   \end{proposition}

   For fixed $n$, write $\theta=\{x/(n-1)\}$, $L_n=x/(n(n-1))$, $\alpha=\{L_n\}$, and $\lambda_n=L_{n-1}-L_n$.
   Since $x/n=x/(n-1)-L_n$ and $x/(n-2)=x/(n-1)+L_{n-1}$, it follows that
   $$ c_n=\lfloor\theta-\alpha\rfloor+\lfloor\theta+\alpha+\lambda_n\rfloor,
      \qquad
      \lambda_n=\frac{2x}{n(n-1)(n-2)}. $$
   Away from the lower end of $I_x$, $\lambda_n$ is small.
   Ignoring it first, if $\alpha\leq 1/2$, then $c_n=-1$ when $\theta<\alpha$, $c_n=1$ when $\theta\geq 1-\alpha$, and otherwise $c_n=0$.
   For $\alpha>1/2$, the same conclusion holds with $\alpha$ replaced by $1-\alpha$.
   Treating $\theta$ as uniform, the probabilities of corrections $-1$ and $1$ are therefore both $\min(\alpha,1-\alpha)=\|L_n\|$.
   Thus the common main contribution satisfies
   $$ C_1(x)+C_{-1}(x)\approx 2\sum_{n\in I_x}\left\|\frac{x}{n^2}\right\|
      \approx x^{1/2}\int_0^\infty t^{-3/2}\|t\|\,dt
      =2Ax^{1/2}. $$
   Evaluating the integral shows
   $$ A=\frac12\int_0^\infty t^{-3/2}\|t\|\,dt
      =\frac{4-\sqrt2}{2\pi}\zeta\!\left(\frac32\right). $$

   The correction $2$ can occur only near the lower end of $I_x$, where $\lambda_n$ is not negligible.
   For fixed $\alpha$, the relevant interval for $\theta$ is nonempty only for $\alpha>1-\lambda_n$.
   As $\alpha$ runs over this range, the length of the $\theta$ interval traces a triangle of base $\lambda_n$ and height $\lambda_n/2$.
   Treating $\theta$ and $\alpha$ as uniform, the average probability is therefore $\lambda_n^2/4$, so
   $$ C_2(x)\approx \sum_{n\in I_x}\frac{\lambda_n^2}{4}
      \approx x^2\int_{\sqrt[3]{2x}}^\infty n^{-6}\,dn
      =\frac{x^{1/3}}{5\cdot2^{5/3}}. $$
   The estimates above give the common main size of the corrections $1$ and $-1$, but do not separate them.
   Their smaller bias is obtained from the total correction sum.
   Let $d_n=q_{n-1}-q_n$.
   Then $c_n=d_{n-1}-d_n$, and hence
   $$ \sum_{n\in I_x}c_n=d_{a-1}-d_b\approx 2^{-2/3}x^{1/3}, $$
   where $a=\lceil\sqrt[3]{2x}\,\rceil$ and $b=\lfloor x^{2/3}\rfloor$.
   On the other hand, the same sum is $C_1-C_{-1}+2C_2$.
   Thus the estimate for $C_2$ determines the small bias
   $$ C_1-C_{-1}\approx 2^{-2/3}x^{1/3}-2Bx^{1/3}
      =2Dx^{1/3}. $$
   Together with the common main term for $C_1+C_{-1}$, this gives, at the level of this heuristic,
   $$ C_{\pm1}(x)\approx Ax^{1/2}\pm Dx^{1/3}. $$

   \section{Recovered Intermediate Values}
   \label{recoveredappendix}

   This appendix gives the criterion used in Sections \ref{recoverypass} and \ref{recoveredvalues} to identify and count the intermediate values recovered by the Mertens value recovery pass.

   \begin{proposition}
   \label{recoveredcriterion}
   Let $N=\lfloor x/u\rfloor$, and let $P^-(k)$ denote the least prime factor of $k>1$, with $P^-(1)=\infty$.
   For any odd square-free $k\leq N$ satisfying
   $$ P^-(k)k>N, $$
   the Mertens value recovery pass described in Section \ref{recoverypass} also recovers $M(x/k)$.
   \end{proposition}

   \begin{proof}
   For $x/k$, the reduced outer sum used by the implementation is
   \begin{align*}
   M(x/k)
      &=\sum_{\substack{d\leq N/(2k)\\(d,2)=1}}
          \mu(d)\big(S(x/(kd),u)-S(x/(2kd),u)\big)\\
       &+\sum_{\substack{N/(2k)<d\leq N/k\\(d,2)=1}}
          \mu(d)S(x/(kd),u).
   \end{align*}
   Since $d\leq N/k<P^-(k)$, every square-free index $d$ in these sums is coprime to~$k$.
   Also, $k$ is odd and square-free, and the sums include only odd $d$, so the indices $kd$ and $2kd$ appearing above are square-free.
   Thus the corresponding stored values $S(x/(kd),u)$ and $S(x/(2kd),u)$ are present when the back substitution is performed.
   \end{proof}

   The number of recovered intermediate values is asymptotic to $CN$, where
   $$ C=
      \frac{6}{\pi^2}
      \sum_{\substack{p\geq 3\\p\ \mathrm{prime}}}
      \frac{p-1}{p(p+1)}
      \prod_{\substack{q<p\\q\ \mathrm{prime}}}\frac{q}{q+1}
      \approx 0.351537. $$
   To see this, fix an odd prime $p=P^-(k)$.
   The density of square-free integers with least prime factor $p$ is
   $$ \frac{6}{\pi^2}\frac{1}{p+1}
      \prod_{\substack{q<p\\q\ \mathrm{prime}}}\frac{q}{q+1}. $$
   For this fixed $p$, the condition $pk>N$ keeps the part of the range with $k>N/p$, contributing the factor $1-1/p=(p-1)/p$ in the summand above.
   The numerical value of $C$ was estimated by truncating the sum at $p\leq 10^7$ and approximating the omitted tail as follows.

   Mertens' product theorem gives
   $$ \prod_{\substack{q<p\\q\ \mathrm{prime}}}\frac{q}{q+1}
      =
      \prod_{\substack{q<p\\q\ \mathrm{prime}}}\frac{1-1/q}{1-1/q^2}
      \sim \frac{e^{-\gamma}/\log p}{6/\pi^2}, $$
   where $\gamma$ is Euler's constant.
   Consequently
   $$ \frac{6}{\pi^2}\frac{p-1}{p(p+1)}
      \prod_{\substack{q<p\\q\ \mathrm{prime}}}\frac{q}{q+1}
      \sim \frac{e^{-\gamma}}{p\log p}, $$
   and summing over primes $p>P$ and applying the prime number theorem shows the omitted tail is asymptotic to
   $e^{-\gamma}/\log P$.

   \section{Helfgott-Thompson Implementation Notes}
   \label{htappendix}

   The comparison in Section \ref{htcompare} uses a separate implementation of the Helfgott-Thompson algorithm based on the authors' original code.
   The goal was to make the timing comparison meaningful on the record machine, rather than to develop a new variant of the algorithm.
   The optimized code reproduces the authors' published values at every input where both have been run.
   More detailed implementation notes are included in \texttt{MertensHT/OPTIMIZATIONS.md} in the code release.

   The main arithmetic changes reduce repeated divisions in three places.
   In \texttt{SArr} and in the large-non-free phase, long runs of nearby quotients use the Quotient Predictor described in Section \ref{quotientcomp}.
   In \textsc{SumByLin}, several quotients and remainders are carried forward incrementally instead of recomputed by division inside the inner loops.
   Several integer arithmetic routines also use 64-bit fast paths when the operands fit in a machine word, falling back to 128-bit arithmetic otherwise.
   These are performance changes only, and every predicted quotient is verified before use.

   A factorization stencil of period $360360=2^3\cdot 3^2\cdot 5\cdot 7\cdot 11\cdot 13$ is precomputed once and reused.
   This removes repeated setup work in the factorization sieve.
   Buffer allocations in repeated routines were also moved outward where possible, arrays storing M\"obius values were changed to signed bytes, and arrays that are immediately overwritten use ordinary allocation rather than zero initialization.
   Integer square roots and related root computations were moved from GMP calls in hot paths to native integer routines with exact correction.

   A further optimization addresses the factorization sieve in the large-non-free phase.
   The array \texttt{SArr} requires the factorization of each integer in a range of length $\Delta+1$, in the notation of \cite{HT}, using primes up to $\sqrt{r_0+\Delta}$.
   At large $x$, a direct pass over these primes scatters updates across more data than fits in L1 cache.
   The implementation instead applies primes in tiles of $4096$ entries.
   Dense prime hits are processed tile by tile, while larger primes are handled by a small local bucket scheduler.
   Below a threshold range length, the overhead of this organization exceeds its benefit and the implementation falls back to a direct pass.

   The parallel structure was also changed.
   In the original code, the large-free phase consisted of nested loops that called the double-sum routine sequentially, with parallelism mainly inside each call.
   The version used here first enumerates many work items and then dispatches them dynamically.
   Many small double-sum items are processed with outer-level parallelism, while the few large brute-force items retain their own internal parallelism.
   Repeated local M\"obius sieves in the double-sum phase are replaced by calls to the segmented M\"obius sieve described in Section \ref{sieve}, so that each relevant interval is sieved once before the parallel summation.

   The double-sum routine in the large-free phase received separate attention.
   Profiling on the record machine singled out this routine, as discussed in Section \ref{htcompare}.
   Temporary arrays in \textsc{SumByLin} are reused through thread-local scratch buffers.
   More importantly, several correction terms are skipped when their intervals are provably empty or when the expression is an exact cancellation of the form
   $\operatorname{sumInter}(J)-\operatorname{sumInter}(J)$.
   In the most common cancellation case, the implementation detects full interval coverage by evaluating the quadratic at the two endpoints of the summation interval, and the inequality is never solved.
   A smaller change to the boundary-correction term skips empty interval sums and avoids recomputing residue products.

   For high-decade runs, the 256-bit fallbacks of Section \ref{htcompare} leave the ordinary 64-bit and 128-bit paths in place whenever their preconditions are satisfied.
   Together with the residue reduction before the multiplication by $n_\circ$, this addresses the arithmetic obstruction noted in \cite{HT} for the projected $10^{24}$ and $10^{25}$ runs.

   These optimizations do not introduce a new reduction in the same spirit as the inclusion-exclusion identities of Section \ref{ie}.
   In the classical implementation, those identities cancel large families of terms over long one-dimensional summation ranges.
   The Helfgott-Thompson computation is organized around many smaller two-dimensional regions and auxiliary \textsc{SumByLin} evaluations, where no comparable global cancellation was found.
   The changes here are therefore implementation optimizations of the existing Helfgott-Thompson decomposition.

   \section{GPU Variants of the Isolated-Value Computations}
   \label{gpuvariant}

   GPU acceleration has precedent here, as Kuznetsov's computation of $M(10^{22})$ was performed on a GPU \cite{EK}.
   Since the aim of this paper includes understanding how hardware shapes the choice of method, GPU variants were built for both isolated-value implementations.
   This appendix describes what was offloaded and when the offload pays.

   On the tested machines, the GPU is not faster than the optimized CPU code at the offloaded arithmetic itself.
   The gain comes only from overlap.
   If the GPU finishes its work while the CPU is still busy with independent work, most of the GPU time is hidden.
   If the CPU reaches the synchronization point first and waits, the advantage disappears.
   This principle governs both variants below.

   In the classical implementation, the only promising target is the 64-bit part of $S_2$.
   The $S_1$ sum consists of noncontiguous lookups into the compressed Mertens arrays, and measured GPU gather throughput was too low for that path to be useful.
   Arguments requiring 128-bit accumulators also remain on the CPU.
   The GPU therefore receives the 64-bit $S_2$ work in bounded waves while the CPU continues with the sieve and $S_1$.
   The waves keep dispatch and staging costs from becoming a significant part of the calculation.
   Enabling the GPU also changes the tuning: the best $c_x$ decreases, shifting visible work from $S_2$ to the threaded $S_1$ sum.
   The direction of this shift may seem backwards, but it follows from the overlap principle.
   The best parameters make the two sides finish at nearly the same time, and since the GPU is slower at this arithmetic, balance is reached by giving it less work.

   In the Helfgott-Thompson implementation, direct offloads of the large-non-free divisor-sum work were correct but slower than the CPU.
   The optimized CPU path is already division-free, cache-resident, and well parallelized, leaving little for the GPU to win back.
   The competitive mode instead overlaps the two phases.
   The brute-force double sums of the large-free phase are sent to the GPU while the CPU runs the large-non-free phase.
   The Diophantine large-free work remains on the CPU.
   The GPU kernels use 64-bit arithmetic, so the offload splits by quotient size: terms with quotients below $2^{60}$ are computed on the GPU, and the few terms above that bound stay on the CPU's 128-bit path.
   This split keeps the GPU engaged well past $10^{18}$.

   The Metal backend used does not expose double precision or native 64-bit integer division.
   On the Apple GPUs, a float-seeded division with exact integer correction removed most of this cost.
   On the AMD GPUs tested, the compiler-provided 64-bit division was faster.
   All offloaded sums are exact, and the GPU variants reproduce the known values of Section \ref{validation} on every machine tested.

   Timings were collected on three machines: the record and laptop machines of Section \ref{isolatedresults}, and an \emph{x86 machine} with a $28$-core Xeon W and a Radeon Pro Vega II Duo, whose two GPUs split the offloaded work.
   Table \ref{table:gputimings} gives the classical timings.
   The GPU-assisted variant is fastest on all three machines at $10^{16}$.
   The largest gain is on the x86 machine, whose CPU is slowest relative to the combined throughput of its two GPUs.
   The advantage then fades as $x$ grows and less of the GPU work is hidden behind the sieve and $S_1$.
   On the Apple machines it is gone by $10^{17}$.
   On the x86 machine it persists through $10^{18}$ and is essentially gone at $10^{19}$.

   \begin{table}[ht]
   \centering
   \begin{tabular}{ | c || r | r || r | r || r | r | }
   \hline
   \multirow{2}{*}{$x$} & \multicolumn{2}{c||}{record machine} & \multicolumn{2}{c||}{laptop machine} & \multicolumn{2}{c|}{x86 machine}\\
   \cline{2-7}
    &  \multicolumn{1}{c|}{CPU} &  \multicolumn{1}{c||}{GPU} &  \multicolumn{1}{c|}{CPU} &  \multicolumn{1}{c||}{GPU} &  \multicolumn{1}{c|}{CPU} & \multicolumn{1}{c|}{GPU}\\
   \hline
   $10^{16}$ & $2.39$ s & {\bf $\boldsymbol{2.38}$ s} & $6.79$ s & {\bf $\boldsymbol{5.94}$ s} & $7.62$ s & {\bf $\boldsymbol{4.97}$ s}\\
   $10^{17}$ & {\bf $\boldsymbol{10.8}$ s} & $14.0$ s & {\bf $\boldsymbol{31.9}$ s} & $37.3$ s & $36.0$ s & {\bf $\boldsymbol{27.7}$ s}\\
   $10^{18}$ & {\bf $\boldsymbol{49.9}$ s} & $69.1$ s & {\bf $\boldsymbol{152}$ s} & $189$ s & $168$ s & {\bf $\boldsymbol{142}$ s}\\
   $10^{19}$ & {\bf $\boldsymbol{238}$ s} & $343$ s & {\bf $\boldsymbol{734}$ s} & $946$ s & {\bf $\boldsymbol{793}$ s} & $795$ s\\
   \hline
   \end{tabular}
   \captionsetup{skip=2pt}
   \caption{Timings with and without the classical GPU $S_2$ offload. Each entry is the best observed end-to-end time at per-machine, per-decade tuned parameters. Boldface marks the faster CPU/GPU run on each machine.}
   \label{table:gputimings}
   \end{table}

   Table \ref{table:gpuht} gives the Helfgott-Thompson overlap timings on the record and x86 machines.
   These two machines represent the two relevant regimes: a fast CPU with unified memory, and a slower CPU paired with discrete GPUs.
   On the record machine, the overlap is slower at every measured decade.
   On the x86 machine, it gives a noticeable speedup through $10^{19}$ and is slower at $10^{20}$, where too little of the GPU's work is hidden behind the remaining CPU phases.

   \begin{table}[ht]
   \centering
   \setlength{\tabcolsep}{4pt}
   \begin{tabular}{ | c || r | r || r | r | }
   \hline
   \multirow{2}{*}{$x$} & \multicolumn{2}{c||}{record machine} & \multicolumn{2}{c|}{x86 machine}\\
   \cline{2-5}
    & \multicolumn{1}{c|}{CPU} & \multicolumn{1}{c||}{GPU} & \multicolumn{1}{c|}{CPU} & \multicolumn{1}{c|}{GPU}\\
   \hline
   $10^{16}$ & {\bf $\boldsymbol{12.74}$ s} & $15.55$ s & $74.84$ s & {\bf $\boldsymbol{49.0}$ s}\\
   $10^{17}$ & {\bf $\boldsymbol{59.68}$ s} & $73.92$ s & $322.3$ s & {\bf $\boldsymbol{217}$ s}\\
   $10^{18}$ & {\bf $\boldsymbol{276.7}$ s} & $330.5$ s & $1349$ s & {\bf $\boldsymbol{919}$ s}\\
   $10^{19}$ & {\bf $\boldsymbol{1240}$ s} & $1462$ s & $1.60$ h & {\bf $\boldsymbol{1.10}$ h}\\
   $10^{20}$ & {\bf $\boldsymbol{1.54}$ h} & $2.55$ h & {\bf $\boldsymbol{6.86}$ h} & $7.59$ h\\
   \hline
   \end{tabular}
   \captionsetup{skip=2pt}
   \caption{Timings with and without Helfgott-Thompson GPU overlap. Boldface marks the faster CPU/GPU run on each machine.}
   \label{table:gpuht}
   \end{table}
   \FloatBarrier

   The two algorithms show the same pattern.
   The GPU helps when it supplies otherwise idle arithmetic capacity on a machine whose CPU stays busy with independent work.
   It helps least on the fastest unified-memory host, where the CPU work is already fast and the GPU competes with the CPU for memory bandwidth.
   The x86 machine benefits more because its discrete GPUs have their own memory and because its CPU is slower at the integer arithmetic being hidden.
   Even there, the advantage does not reach the largest inputs tested.
   
   These are partial offloads, so the possible gains have an Amdahl-type ceiling.
   Larger gains would require moving additional phases to the GPU or shifting the tuning substantially toward the offloaded work.
   On hardware exposing double precision or native 64-bit integer division, as in the CUDA setting of \cite{EK}, the emulation costs paid here disappear and the measured constants could change.
   On the machines used here, however, the GPU variants are useful validation and comparison tools rather than improvements to the record-scale path.



\begin{thebibliography}{99}

   \bibitem{CD}
     H. Cohen and F. Dress, \emph{Calcul num\'erique de $M(x)$}
     Rapport de l'ATP A12311 ``Informatique 1975'', CNRS (1979) 1\,--\,13.

   \bibitem{DR}
     M. Del\'eglise and J. Rivat, \emph{Computing the summation of the M\"obius function}
     Exp. Math. {\bf 5} (1996) 291\,--\,295.

   \bibitem{Dre}
     F. Dress, \emph{Fonction sommatoire de la fonction de M\"obius, 1. Majorations exp\'erimentales}
     Exp. Math. {\bf 2} (1993) 89\,--\,98.

   \bibitem{GM}
     T. Granlund and P. Montgomery, \emph{Division by invariant integers using multiplication}
     Proceedings of the ACM SIGPLAN 1994 Conference on Programming Language Design and Implementation, 1994.

   \bibitem{HelSieve}
     H. A. Helfgott, \emph{An improved sieve of Eratosthenes}
     Math. Comp. {\bf 89} (2020) 333\,--\,350.

   \bibitem{HT}
     H. A. Helfgott and L. Thompson, \emph{Summing $\mu(n)$: a faster elementary algorithm}
     Res. Number Theory {\bf 9} (2023), no. 1, Paper No. 6.

   \bibitem{HKM}
     D. Hirsch, I. Kessler, and U. Mendlovic, \emph{Computing $\pi(N)$: an elementary approach in $\widetilde{O}(\sqrt{N})$ time}
     arXiv:2212.09857 [math.NT] (2023).

   \bibitem{Hur18}
     Greg Hurst, \emph{Computations of the Mertens function and improved bounds on the Mertens conjecture}
     Math. Comp. {\bf 87} (2018) 1013\,--\,1028.

   \bibitem{KL}
     T. Kotnik and J. van de Lune, \emph{Further systematic computations on the summatory function of the M\"obius function}
     Report MAS-R0313, CWI Amsterdam (November 2003).

   \bibitem{EK}
     Eugene Kuznetsov, \emph{Computing the Mertens function on a GPU}
     arXiv:1108.0135 [math.NT] (2011).

   \bibitem{LO}
     J. C. Lagarias and A. M. Odlyzko, \emph{Computing $\pi(x)$: an analytic method}
     J. Algorithms {\bf 8} (1987) 173\,--\,191.

   \bibitem{Leh}
     R. S. Lehman, \emph{On Liouville's function}
     Math. Comp. {\bf 14} (1960) 311\,--\,320.

   \bibitem{LL}
     W. M. Lioen and J. van de Lune, \emph{Systematic computations on Mertens' conjecture and Dirichlet's divisor problem by vectorized sieving}
     From Universal Morphisms to Megabytes: a Baayen Space Odyssey, CWI Amsterdam (1994) 421\,--\,432.

   \bibitem{Mer}
     F. Mertens, \emph{\"Uber eine zahlentheoretische Funktion}
     Sitzungsberichte Akad. Wiss. Wien IIa {\bf 106} (1897) 761\,--\,830.

   \bibitem{Neu}
     G. Neubauer, \emph{Eine empirische Untersuchung zur Mertensschen Funktion}
     Numer. Math. {\bf 5} (1963) 1\,--\,13.

   \bibitem{NG}
     Nathan Ng, \emph{The distribution of the summatory function of the M\"obius function}
     Proc. London Math. Soc. {\bf 89} (2004) 361\,--\,389.

   \bibitem{OtR}
     A. M. Odlyzko and H. J. J. te Riele, \emph{Disproof of the Mertens conjecture}
     J. Reine Angew. Math. {\bf 357} (1985) 138\,--\,160.

   \bibitem{St1}
     R. D. von Sterneck, \emph{Empirische Untersuchung \"uber den Verlauf der zahlentheoretischen Funktion $\sigma(n)=\sum_{x=1}^{n}\mu(x)$ im Intervalle von 0 bis 150\,000}
     Sitzungsberichte Akad. Wiss. Wien IIa {\bf 106} (1897) 835\,--\,1024.

   \bibitem{St2}
     R. D. von Sterneck, \emph{Neue empirische Daten \"uber die zahlentheoretische Funktion $\sigma(n)$}
     Proc. 5th International Congress of Mathematicians, vol. 1, Cambridge University Press (1913) 341\,--\,343.

   \end{thebibliography}
\end{document}